\documentclass[11 pt, leqno]{amsart}

\usepackage{amsmath, amsthm, amssymb, amsfonts}
\usepackage{latexsym}
\usepackage[active]{srcltx} 
\usepackage[all]{xy}

\usepackage{a4wide}
\usepackage{comment}
\usepackage{graphicx}
\usepackage{dsfont}

\theoremstyle{plain}
\newtheorem{tw}{Theorem}[section]
\newtheorem {lem} [tw]{Lemma}
\newtheorem {prop}[tw] {Proposition}

\newtheorem{cor}[tw]{Corollary}

\newtheorem*{theorem*}{Theorem}

\theoremstyle{definition}
\newtheorem {deft}[tw] {Definition}
\newtheorem {rem} [tw]{Remark}
\newtheorem{exam}[tw]{Example}

\newcommand{\bc}{\mathbb C}
\newcommand{\bn}{\mathbb N}
\newcommand{\br}{\mathbb R}
\newcommand{\Irr}{\textup{Irr}}
\newcommand{\Pol}{\textup{Pol}}
\newcommand{\bt}{\mathbb T}
\newcommand{\bz}{\mathbb Z}

\newcommand{\alg} {\mathsf{A}}
\newcommand{\blg} {\mathsf{B}}

\newcommand {\id} {{\operatorname{id}}}
\newcommand{\Alg} {\mathcal{A}}
\newcommand{\wot} {\overline{\ot}}

\newcommand{\tu}{\textup}

\newcommand{\Hil}{\mathsf{H}}

\newcommand{\QG}{\mathbb{G}}
\newcommand{\RepGH}{\tu{Rep}_{\QG}(\Hil)}

\newcommand{\Cou}{\epsilon}
\newcommand{\Com}{\Delta}
\newcommand{\Ind}{\mathcal{I}}
\newcommand{\Ker}{\tu{Ker}\,}

\newenvironment{rlist}
{

\begin{enumerate}}
{\end{enumerate}}

\newcommand{\M}{\mathsf{M}}

\newcommand{\ot}{\otimes}

\numberwithin{equation}{section}

\newcommand{\ZZ}{\mathbb{Z}}

\newcommand{\hh}[1]{\widehat{#1}}
\newcommand{\G}{\mathbb{G}}
\newcommand{\QH}{\mathbb{H}}
\renewcommand{\H}{\mathbb{H}}
\newcommand{\tp}{\xymatrix{*+<.7ex>[o][F-]{\scriptstyle\top}}}
\newcommand{\rep}{\operatorname{Rep}}
\newcommand{\wkcon}{\preccurlyeq}

\newcommand{\Wuniv}{\mathds{W}}

\newcommand{\sH}{\mathsf{H}}

\newcommand{\mc}{\mathcal}
\newcommand{\mf}{\mathfrak}
\newcommand{\ip}[2]{\langle #1,#2 \rangle}
\newcommand{\op}{{\operatorname{op}}}

\newcommand{\WW}{{\mathds{V}\!\!\text{\reflectbox{$\mathds{V}$}}}}
\newcommand{\Ww}{\mathds{W}}
\newcommand{\wW}{\text{\reflectbox{$\Ww$}}\:\!} 


\begin{document}

\author{Matt Daws}
\address{School of Mathematics, University of Leeds, Leeds LS2 9JT, United Kingdom}
\email{matt.daws@cantab.net}
\thanks{MD is partially supported by EPSRC grant EP/IO26819/1.}
\author{Pierre Fima}
\address{Universit\'e Denis-Diderot - Paris 7, Institut Math\'ematiques de Jussieu,  175, rue du Chevaleret, 75 013 Paris, France}
\email{pfima@math.jussieu.fr}
\thanks{PF is partially supported by the ANR grants NEUMANN and OSQPI}
\author{Adam Skalski}
\address{Institute of Mathematics of the Polish Academy of Sciences,
ul.~\'Sniadeckich 8, 00--956 Warszawa, Poland \newline \indent Faculty of Mathematics, Informatics and Mechanics, University of Warsaw, ul.~Banacha 2,
02-097 Warsaw, Poland}
\thanks{AS is partially supported by the Iuventus Plus grant IP2012 043872.}
\email{a.skalski@impan.pl}
\author{Stuart White}
\address{ School of Mathematics and Statistics, University of
Glasgow, University Gardens, Glasgow, G12 8QW, Scotland}
\email{stuart.white@glasgow.ac.uk}
\thanks{SW is partially supported by EPSRC grant EP/IO19227/1-2.}

\title{The Haagerup property for locally compact quantum groups}

\begin{abstract}
The Haagerup property for locally compact groups is generalised to the context of locally compact quantum groups, with several equivalent characterisations in terms of the unitary representations and positive-definite functions established. In particular it is shown that a locally compact quantum group $\QG$ has the Haagerup property if and only if its mixing representations are dense in the space of all unitary representations. For discrete $\QG$ we characterise the  Haagerup property by the existence of a symmetric proper conditionally negative functional on the dual quantum group $\widehat{\QG}$; by the existence of a real proper cocycle on $\QG$, and further, if $\QG$ is also unimodular we show that the Haagerup property is a von Neumann property of $\QG$. This extends results of Akemann, Walter, Bekka, Cherix, Valette, and Jolissaint to the quantum setting
and provides a connection to the recent work of Brannan. We use these characterisations to show that the Haagerup property is preserved under free products of discrete quantum groups.
\end{abstract}

\maketitle

The Haagerup property of locally compact groups has its origins in Haagerup's fundamental paper \cite{Haagerup}, which establishes that the length function on the free group  $\mathbb{F}_n$ is conditionally negative-definite and uses this to obtain approximation properties for the reduced $C^*$-algebra, $C^*_r(\mathbb {F}_n)$. For a locally compact group $G$, the key ingredient in \cite{Haagerup} is now known to be equivalent to several conditions (see \cite{AW,Jol0,book}) which define the \emph{Haagerup property}:
\begin{itemize}
\item $G$ has the Haagerup property if it admits a mixing unitary representation which weakly contains the trivial representation;
\item $G$ has the Haagerup property if there is a proper, continuous conditionally negative definite function $G\rightarrow\mathbb C$;
\item $G$ has the Haagerup property if there is a normalised sequence of continuous, positive definite functions vanishing at infinity which converges uniformly to $1$ on compact subsets of $G$;
\item $G$ has the Haagerup property if there is a proper continuous affine action of $G$ on a real Hilbert space, or equivalently, $G$ admits a proper continuous cocycle for some representation of $G$ on a real Hilbert space.
\end{itemize}

The Haagerup property can also be defined for von Neumann algebras (see \cite{j}, for example) and, for  discrete $G$, the Haagerup property is a von Neumann property: $G$ enjoys the Haagerup property precisely when the group von Neumann algebra $VN(G)$ does (\cite{choda}).

The Haagerup property is often interpreted as a weak form of amenability. Indeed, the left regular representation of an amenable group is mixing (i.e. its matrix coefficients vanish at infinity) and weakly contains the trivial representation, so amenable groups have the Haagerup property.   Other examples include free groups (via Haagerup's original paper \cite{Haagerup}), finitely generated Coxeter groups, $SL(2,\bz)$, $SO(1,n)$ and $SU(1,n)$; moreover the Haagerup property is preserved by free products.  Striking applications include Higson and Kasparov's proof of the Baum-Connes conjecture in the presence of the Haagerup property (\cite{HK}), and Popa's deformation-rigidity approach to structural properties of type II$_1$ factors (\cite{Po1,Po2}).  We refer to  \cite{book} for the equivalence of the formulations above, examples and applications.

In this paper we undertake a systematic study of the Haagerup property in the setting of locally compact quantum groups.   The theory of topological quantum groups has developed rapidly in the last twenty years through  Woronowicz's definition of a compact quantum group (\cite{wor}) and the locally compact quantum group framework of Kustermans and Vaes (\cite{KV}). Using the language of operator algebras, the latter theory  generalises the classical theory of locally compact groups, providing a full extension of the Pontryagin duality for locally compact abelian groups, and encompasses a large class of examples.  A ``locally compact quantum group'' $\QG$ is studied via its ``algebras of functions'': the von Neumann algebra $L^{\infty}(\QG)$ and the $C^*$-algebra $C_0(\QG)$ correspond to the algebras of essentially bounded (with respect to the Haar measure) functions on a locally compact group $G$ and the continuous functions on $G$ which vanish at infinity respectively. As the theory reached a certain level of maturity, it became natural to investigate questions relating quantum groups to noncommutative probability, noncommutative geometry, and analyse actions of quantum groups. In particular a study of approximation-type/geometric properties such as amenability (\cite{BT} and references there) or property (T) (\cite{PierreT}, \cite{ks}) has recently been initiated, often with a special focus on the case of discrete quantum groups. Recently Brannan established the Haagerup property for the von Neumann algebras associated to free orthogonal and unitary quantum groups \cite{Brannan} (analogous results have also been obtained for the von Neumann algebras associated to certain quantum automorphism groups \cite{Brannan2}, and quantum reflection groups \cite{Lemeux}). Within the theory of quantum groups, $VN(G)$ is interpreted as the algebra $L^{\infty}(\widehat{G})$, where $\widehat{G}$ is the dual quantum group of $G$. Thus, by analogy with the fact that the Haagerup property is a von Neumann property of a discrete group \cite{choda}, Brannan's result can be viewed as the statement that the discrete dual quantum groups of the free orthogonal and unitary quantum groups have the Haagerup property. It was shown in \cite{Voigt}  that the duals of  free orthogonal groups satisfy an appropriate version of the Baum-Connes conjecture.

The papers \cite{Brannan,Brannan2} study the Haagerup approximation property for $L^\infty(\G)$, where $\G$ is a discrete unimodular quantum group. However it would not be satisfactory to define the Haagerup property for general $\hh\G$ by asking for $L^\infty(\G)$ to have the Haagerup property: this is not phrased intrinsically in terms of properties of $\hh\G$, and, more importantly, is problematic when $L^\infty(\G)$ does not admit a finite faithful trace.  Thus our starting point is the classical definition in terms of the existence of a mixing representation weakly containing the trivial representation.  We translate this into the quantum setting and show that it is equivalent to the existence of an approximate identity for $C_0(\QG)$ arising from completely positive multipliers.  We set out how the Haagerup property can be viewed through the lens of global properties of representations: when $\QG$ is second countable, the Haagerup property is equivalent to the density of mixing representations in the collection of all unitary representations of $\QG$ on a fixed infinite dimensional separable Hilbert space.  This extends the philosophy of \cite{Kechris} to the quantum setting and generalises work of Bergelson and Rosenblatt \cite{BR}.

In the case when $\QG$ is discrete we give three further characterisations of the Haagerup property, summerised below. The precise terminology will be described in Sections \ref{sec:haaprop} and \ref{disc} and the theorem is obtained by combining Theorems~\ref{groupalg}, \ref{groupalgKac}, \ref{generators} and ~\ref{HAP-cocycle}. Note that we do not require $\QG$ to be unimodular in the equivalence of (i)-(iv).
\begin{theorem*}
Let $\QG$ be a discrete quantum group. Then the following conditions are equivalent:
\begin{rlist}
\item $\QG$ has the Haagerup property;
\item there exists a convolution semigroup of states $(\mu_t)_{t\geq 0}$ on $C_0^u(\widehat{\QG})$ such that each $a_t:=(\mu_t \ot \id)(\hh\Wuniv)$ is an element of $C_0(\QG)$ and $a_t$ tends strictly to $1$ as $t\to 0^+$;
\item $\widehat{\QG}$ admits a symmetric proper generating functional;
\item $\QG$ admits a proper real cocycle.
\end{rlist}
Further, (i)-(iv) imply the following condition, and, for unimodular $\QG$,
are equivalent to it:
\begin{rlist}
\item[(v)] $L^{\infty}(\hh \QG)$ has the Haagerup approximation property.
\end{rlist}
\end{theorem*}

The equivalence of (i) and (iii) extends a classical result of Jolissaint on conditionally negative-definite functions; the equivalence of (i) and (iv) is a quantum version of the result of Bekka, Cherix and Valette from \cite{bcv}, which is the starting point of the geometric interpretations of the Haagerup property. In the unimodular case, the final statement generalises Choda's work and justifies our interpretation of the results in \cite{Brannan} described above.

In the last section we use the characterisations obtained earlier, together with the theory of conditionally free products of states, to prove that the Haagerup property is preserved under taking free products of discrete quantum groups.  The techniques used in  this article are based on the analysis of various properties of unitary representations of locally compact quantum groups, on applications of completely positive multipliers of locally compact quantum groups, as studied for example in \cite{jnr} and \cite{daws}, and on certain results concerning the convolution semigroups of states on compact quantum groups and their generating functionals (see \cite{LS}).

The detailed plan of the paper is as follows: Section 1 introduces notation and terminology, and proves an $L^2$-implementation results for certain completely positive multipliers. Section 2 is devoted to the analysis of containment and weak containment of representations of a given locally compact quantum group $\QG$; then in a short Section 3 we define, by analogy with the classical context, mixing representations and set out some of their properties. In Section 4 we equip the space of representations of $\QG$ on a fixed Hilbert space with a Polish topology. Section 5 introduces the Haagerup property for a locally compact quantum group and presents the first part of the main results of the paper (Theorem \ref{equivHAP}). In Section 6 we specialise to discrete quantum groups and prove the theorem above. Finally in Section 7 we apply the earlier results and use a construction of conditionally free products of states to prove that the Haagerup property is preserved under free products of discrete quantum groups. We also give some generalisations of the last result concerning free products with amalgamation over a finite quantum subgroup and  quantum HNN extensions.

\subsection*{Acknowledgements} Some work on this paper was undertaken during a visit of AS and SW to the University of Leeds in June 2012, funded by EPSRC grant EP/I026819/I. They thank the faculty of the School of Mathematics for their hospitality.  The authors would also like to thank Jan Cameron, Caleb Eckhardt, David Kyed, Roland Vergnioux and the anonymous referee for valuable comments and advice.

\section{Notation, terminology and some technical facts} \label{not}

Scalar products (both for Hilbert spaces and Hilbert modules) will be linear on the left. The symbol $\ot$ denotes the spatial tensor product of $C^*$-algebras, and if $\alg$ is a $C^*$-algebra then $\M(\alg)$ denotes its multiplier algebra; a \emph{morphism} between $C^*$-algebras $\alg$ and $\blg$ is a nondegenerate $^*$-homomorphism from $\alg$ to $\M(\blg)$, and we write $\textup{Mor}(\alg,\blg)$ for the collection of all these morphisms (see \cite{lance} for a treatment of multiplier $C^*$-algebras and corresponding morphisms). For a Hilbert space $\Hil$ the symbol $\mc K(\Hil)$ will denote the algebra of compact operators on $\Hil$ and if $\xi, \eta \in \Hil$, then $\omega_{\xi,\eta}\in \mc K(\Hil)^*$ will be the vector functional,
$T\mapsto (T\xi|\eta)$.

For a $C^*$-algebra $\alg$ and a Hilbert space $\sH$, the  Hilbert
$C^*$-module $\alg\otimes\sH$ is the completion of the algebraic tensor
product $\alg \odot \sH$ in the $\alg$-valued inner-product $(a\otimes\xi|b\otimes\eta)
= (\xi|\eta) b^*a$, ($a,b \in \alg$, $\xi, \eta \in \Hil$).
Write $\mc L(\alg\otimes\sH)$ for the space of ``adjointable maps'' on $\alg \ot \sH$ (see \cite{lance}).

We will often use the canonical isomorphism between $\M(\alg\otimes\mc K(\Hil))$
and $\mc L(\alg\otimes\sH)$ and think of it as mapping an element  (usually a unitary) $U\in \M(\alg\otimes\mc K(\Hil))$ to an adjointable map $\mc U\in\mc L(\alg\otimes\sH)$, the relation being
\begin{equation}
(\mc U(a\otimes\xi) | b\otimes\eta)
= b^* (\id\otimes\omega_{\xi,\eta})(U) a,
\qquad (a,b\in \alg, \xi,\eta\in \sH).  \label{eq:mattthree}
\end{equation}

As is standard in the theory of quantum groups, we use the `leg' notation for operators acting on tensor products of Hilbert spaces or $C^*$-algebras.

\subsection{Locally compact quantum groups}\label{lcqgssection}

For the theory of locally compact quantum groups we refer the reader to \cite{KV} and to the lecture notes \cite{kus} (in particular we will use the conventions of these papers).  For a locally compact quantum group $\QG$ the corresponding $C^*$-algebra of ``continuous functions on $\QG$ vanishing at infinity'' will be denoted by $C_0(\QG)$. It is equipped with a \emph{comultiplication} (or \emph{coproduct}) $\Delta\in\textup{Mor}\bigl(C_0(\QG),C_0(\QG)\ot C_0(\QG)\bigr)$ and left and right Haar  weights, $\varphi$ and $\psi$. The \emph{dual locally compact quantum group} of $\QG$ will be denoted by $\hh \QG$. We usually assume (using the GNS construction with respect to the left invariant Haar weight) that both $C_0(\QG) $ and $C_0(\hh \QG)$ act on the Hilbert space $L^2(\QG)$.  The
fundamental multiplicative unitary $W \in \M(C_0(\QG)\otimes C_0(\hh\QG))$ implements the comultiplication by $\Delta(x)=W^*(1\otimes x)W$ for all $x\in C_0(\QG)$. The ``universal'' version of $C_0(\QG)$ (see \cite{ku}) will be denoted by $C_0^u(\QG)$, with the canonical \emph{reducing morphism}
$\Lambda_{\G}:C_0^u(\G)\rightarrow C_0(\G)$ and the \emph{counit} $\epsilon_u:C_0^u(\G) \to \bc$.  The von Neumann algebra generated by $C_0(\QG)$ in $\mc B(L^2(\QG))$ will be denoted by $L^{\infty}(\QG)$ and the densely defined \emph{antipode} by $S$; it maps ${\mc D}_S\subset C_0(\G)$ into $C_0(\G)$.  We shall occasionally use the strict extension of $S$ to $\M (C_0(\G))$ which has domain $\overline{{\mc D}}_S$, see \cite{kus3}. The \emph{unitary antipode} of $\QG$, which is a $^*$-anti-automorphism of $C_0(\QG)$, will be denoted by $R$.   The predual of $L^{\infty}(\QG)$ will be denoted by $L^1(\QG)$.  The pre-adjoint of the comultiplication (which extends to $L^\infty(\QG)\rightarrow L^\infty(\QG)\,\overline{\otimes}\,L^\infty(\QG)$) equips $L^1(\QG)$ with the structure of a completely contractive Banach algebra with product $*$.  The respective maps related to the dual quantum group will be adorned with hats, so that for example the right invariant weight on $\widehat{\QG}$ will be denoted by $\widehat{\psi}$. Similarly the maps acting on the universal algebra $C_0^u(\G)$ are equipped with an index $u$; for example $R_u$ is the universal version of the unitary antipode.

We say that $\QG$ is \emph{coamenable} if the reducing morphism $\Lambda_{\G}$ is an isomorphism and \emph{amenable} if $L^{\infty}(\QG)$ possesses an invariant mean, i.e. a state $m$ on $L^\infty(\QG)$ satisfying
\begin{equation}
m( (\omega\otimes\id)\Delta(x) ) = m( (\id\otimes\omega)\Delta(x) ) = \omega(1)m(x),\qquad (\omega\in L^1(\QG),\ x\in L^\infty(\QG)).
\end{equation} If the left and right Haar weights of $\QG$ coincide, we say that $\QG$ is \emph{unimodular}. In the case when $\QG$ is \emph{compact} (i.e. $C_0(\QG)$ is unital) recall that $\QG$ is \emph{Kac} if its antipode $S$ is bounded, or equivalently, $\widehat{\QG}$ is unimodular. We say that $\QG$ is finite if $C_0(\QG)$ is finite-dimensional.

For convenience, throughout this paper we will assume that \textbf{all locally compact quantum groups considered are second-countable} (by which we mean that $C_0(\QG)$ is separable).  However, we only make essential use of this assumption in Section~\ref{sec:rep_met_sp}, and again in Section~\ref{sec:semigp_conv_states} (in particular Theorem~\ref{generators}) and in Section~\ref{sec:free_prods}.

As shown by Kustermans in \cite{ku}, the multiplicative unitary $W$ admits a ``semi-universal'' version,
a unitary $\Wuniv \in \M(C_0^u(\QG) \otimes C_0(\hh\QG))$ characterised by the
following ``universal'' property: for a $C^*$-algebra $\blg$, there is a bijection between:
\begin{itemize}
\item unitary elements $U\in \M(\blg \otimes C_0(\hh\QG))$ with
$(\id\otimes\hh\Delta)(U) = U_{13} U_{12}$; and
\item non-degenerate $*$-homomorphisms $\phi_U:C_0^u(\QG)\rightarrow \M(\blg)$,
\end{itemize}
given by $(\phi_U\otimes\id)(\Wuniv) = U$. When the unitary $U$ is fixed, we will sometimes write $\phi$ rather than $\phi_U$.

Similarly, there is a unitary $\wW \in \M(C_0(\QG)\otimes C_0^u(\hh\QG))$,
universal in the sense that for every $C^*$-algebra $\blg$, there is a bijection between:
\begin{itemize}
\item unitary elements $U\in \M(C_0(\G)\otimes \blg)$ with
$(\Delta\otimes\id)(U) = U_{13} U_{23}$; and
\item non-degenerate $*$-homomorphisms $\phi_U:C_0^u(\hh\QG)\rightarrow \M(\blg)$.
\end{itemize}
The bijection is again given by a similar relation:
$(\id\otimes\phi_U)(\wW) = U$.
Note that in \cite{ku},  $\mc V$ and $\hh{\mc V}$ are used for $\Wuniv$ and $\wW$, respectively.

We can also consider  the multiplicative unitary of the dual quantum group $\hh \G$,  $\hh W \in \M(C_0(\QG)\otimes C_0(\hh\QG))$. It is  equal to $\sigma(W)^*$, where $\sigma:\M(C_0(\hh\QG)\otimes C_0(\G)) \rightarrow \M(C_0(\QG) \otimes C_0(\hh\QG))$ is the ``swap map''.

A \emph{Hopf $*$-homomorphism} $\pi:C_0^u(\G)\rightarrow MC_0(\H)$ is a non-degenerate $*$-homomorphism intertwining the coproducts.  In this case there is a unique Hopf $*$-homomorphism $\hh\pi:C_0^u(\hh\H) \rightarrow \M(C_0^u(\hh\G))$ ``dual'' to $\pi$ (see for example \cite{mrw}).  In addition to the unitaries $\Wuniv$ and $\wW$, there is also a truly universal bicharacter $\WW\in\M(C_0^u(\G)\otimes C_0^u(\hh\G))$, which satisfies the conditions $\Ww = (\id\otimes\Lambda_{\hh\G})(\WW)$ and
$\wW = (\Lambda_{\G}\otimes\id)(\WW)$.  Then $\hh\pi$ is uniquely
determined by the relation
\begin{equation}
(\pi\otimes\id)(\WW_{\G}) = (\Lambda_{\H}\otimes\hh\pi)(\WW_{\H})
=(\id\otimes\hh\pi)(\wW_{\H}).
\end{equation}
There are various notions of a ``closed quantum subgroup'' $\H$ of a locally compact quantum group $\G$ in the literature and these are analysed in detail in the recent article \cite{closedqs}.  The weakest of these is when $\pi$ maps surjectively onto $C_0(\H)$, in which case $\H$ is called a closed quantum subgroup of $\G$ \emph{in the sense of Woronowicz}.
If $\H$ is compact (in particular, finite) or $\QG$ is discrete (see Subsection 1.3 below) this notion is equivalent to the notion of closed quantum subgroup \emph{in the sense of Vaes}, i.e.\ to the existence of an injective normal $^*$-homomorphism $\theta:L^{\infty}(\widehat{\H})\to L^{\infty}(\widehat{\G})$, intertwining the respective comultiplications (see Section 6 of \cite{closedqs}). We then say simply that $\QH$ is a closed quantum subgroup of $\QG$.

\subsection{Unitary representations of locally compact quantum groups} \label{unitreplcqg}

Let $\QG$ be a locally compact quantum group.

\begin{deft} A unitary representation of $\QG$ (or a unitary corepresentation
of $C_0(\G)$) on a Hilbert space $\Hil$ is a unitary $U\in \M(C_0(\QG)\otimes
\mathcal \mc K(\Hil))$ with $(\Delta\otimes\id)U = U_{13} U_{23}$.  We will often write $\Hil_U$ for the Hilbert space upon which $U$ acts.
The trivial representation of $\QG$, i.e.\ $U=1 \ot 1 \in  \M(C_0(\QG)\otimes
\mathcal \bc)$, will be denoted simply by $1$.
\end{deft}

The \emph{tensor product} of two representations $U$ and $V$ is
$U \tp V = U_{12} V_{13}$, which acts on $\sH_U \otimes \sH_V$.  The direct sum of two representations is easy to understand, but a little harder to write down.  Formally, let $\iota_V:\sH_V\rightarrow
\sH_V\oplus\sH_U$ be the inclusion, and $p_V:\sH_V\oplus\sH_U \rightarrow \sH_V$
be the projection, and similarly for $\iota_U$ and $p_U$.  Then define
\begin{equation} U \oplus V = (1\otimes\iota_U) U (1\otimes p_U)
+ (1\otimes\iota_V) V (1\otimes p_V) \in
\M(C_0(\QG) \otimes \mc K(\sH_U\oplus\sH_V)).
\end{equation}
The last formula  may seem strange at first sight if one thinks of  multiplier
algebras, but has a natural interpretation in terms of adjointable operators.  Slightly more informally, we first make the identification
\begin{equation} \mc K(\sH_U\oplus\sH_V) = \begin{pmatrix} \mc K(\sH_U) & \mc K(\sH_V,\sH_U) \\
\mc K(\sH_U,\sH_V) & \mc K(\sH_V) \end{pmatrix}.
\end{equation}
Then it is easy to see how we view $\M(C_0(\QG)\otimes\mc K(\sH_U))$ as
a subalgebra of $\M(C_0(\QG)\otimes\mc K(\sH_U\oplus\sH_V))$, basically the ``upper
left corner'', and similarly for $\M(C_0(\QG)\otimes\mc K(\sH_V))$, the
``lower right corner''.

A representation of $\QG$ is called \emph{irreducible} if it is not (unitarily equivalent to) a direct sum of two non-zero representations.

\subsection{Compact/discrete quantum groups} \label{compdiscsubsect}

A locally compact quantum group $\QG$ is called \emph{compact} if the algebra $C_0(\G)$ is unital (we then denote it simply by $C(\G)$), or, equivalently, the Haar weight is in fact a bi-invariant state. It is said to be \emph{discrete} if $C_0(\G)$ is a direct sum of matrix algebras (and is then denoted $c_0(\G)$), or, equivalently, $\widehat{\QG}$ is compact. For a compact quantum group $\QG$  the symbol $\Irr_{\QG}$ will denote the collection of equivalence classes of finite-dimensional unitary representations of $\QG$ (note that our assumptions imply it is a countable set). We will always assume that for each $\alpha\in \Irr_{\QG}$ a particular representative has been chosen and moreover identified with a unitary matrix $U^{\alpha}= (u_{ij}^{\alpha})_{i,j=1}^{n_{\alpha}} \in \mathbb M_{n_{\alpha}} (C(\QG))$. The span of all the \emph{coefficients} $u_{ij}^{\alpha}$ is a dense (Hopf) *-subalgebra of $C(\QG)$, denoted $\Pol(\QG)$. The algebra of  functions vanishing at infinity on the dual discrete quantum group is given by the equality $c_0(\widehat{\QG})= \bigoplus_{\alpha \in \Irr _\QG} \mathbb M_{n_\alpha}$. Thus the elements affiliated to $c_0(\widehat{\QG})$ can be identified with functionals on $\Pol(\QG)$. Note that the Haar state of $\QG$ is faithful on $\Pol(\QG)$; moreover in fact $C^u(\G)$ is the enveloping $C^*$-algebra of $\Pol(\G)$, and thus we can also view the latter algebra as a subalgebra of $C^u(\QG)$. The (semi-)universal multiplicative unitary of $\G$ is then given by the formula
\begin{equation}\label{unitcompact} \Wuniv = \sum_{\alpha \in \Irr_\QG}  u_{ij}^\alpha \ot e_{ij}^{\alpha} \in
\prod_{\alpha \in \Irr_\G} C^u(\G) \otimes \mathbb M_{n_\alpha} = \M(C^u(\G) \ot c_0(\widehat{\QG})). \end{equation}

By a \emph{state} on $\Pol(\G)$ we mean a linear functional $\mu:\Pol(\G)\rightarrow\mathbb C$ which is positive in the sense that $\mu(a^*a)\geq0$ for all $a\in\Pol(\G)$.  We can follow the usual GNS construction to obtain a pre-Hilbert space $\Hil_0$, a cyclic vector $\xi_0\in \Hil_0$ and a $*$-homomorphism $\Pol(\G)\rightarrow \mc L(\Hil_0)$, the collection of adjointable maps on $\Hil_0$, with $\mu(a) = (\pi(a)\xi_0|\xi_0)$.  As argued in \cite[Lemma~4.2]{DK} (compare \cite[Lemma~8.7]{QSCC2}) for the algebra $\Pol(\G)$, the map $\pi$ always extends to a $*$-homomorphism $\Pol(\G)\rightarrow\mc B(\Hil)$, where $\Hil$ is the completion of $\Hil_0$.  As $C^u(\G)$ is the enveloping $C^*$-algebra of $\Pol(\G)$, we see that there is a bijection between states on $C^u(\G)$ and states on $\Pol(\G)$. To simplify the notation we will occasionally write simply $\Cou$ to denote the counit of $\G$ understood as a character on $\Pol(\G)$.

\subsection{Multipliers of quantum groups}\label{sec:mults}

The notion of (completely bounded) multipliers on locally compact quantum groups plays a crucial role in the paper. Here we review the relevant notions and prove a technical result for later use.

\begin{deft} A \emph{completely bounded  left multiplier} of $L^1(\hh\G)$ is a bounded linear map
$L_*: L^1(\hh\G)\rightarrow L^1(\hh\G)$ such that $L_*(\omega_1*\omega_2)
= L_*(\omega_1) * \omega_2$ for all $\omega_1,\omega_2 \in L^1(\hh \G)$, and whose adjoint $L=(L_*)^*$ is a completely
bounded map on $L^\infty(\hh\G)$.
\end{deft}

The adjoints of completely bounded left multipliers can be characterised as follows.

\begin{prop}\label{prop:mult_char}
Let $L:L^\infty(\hh\G)\rightarrow L^\infty(\hh\G)$ be a normal, completely
bounded map.  Then the following are equivalent:
\begin{enumerate}
\item\label{prop:mult_char:one}
$L$ is the adjoint of a left multiplier of $L^1(\hh\G)$;
\item\label{prop:mult_char:two}
$\hh\Delta\circ L = (L\otimes\id)\circ \hh\Delta$;
\item\label{prop:mult_char:three}
there is $a\in L^\infty(\G)$ with
$(L\otimes\id)(\hh W) = (1\otimes a)\hh W$.
\end{enumerate}
If these hold, then actually $a\in \M (C_0(\G))$, and we have that $a\widehat\lambda(\omega) = \widehat\lambda(L_*(\omega))$ for $\omega\in L^1(\hh\G)$, where $\hh\lambda:L^1(\hh\G)\rightarrow C_0(\G)$ is defined by $\hh\lambda(\omega)=(\omega\otimes\id)(\hh W)$ for $\omega\in L^1(\hh\G)$.  In this way, the multiplier $L_*$ is given by left multiplication by the element $a$ in the left regular representation.
\end{prop}
\begin{proof}
Conditions (\ref{prop:mult_char:one}) and (\ref{prop:mult_char:two}) are
easily seen to be equivalent, and (\ref{prop:mult_char:three}) implies
(\ref{prop:mult_char:one}) is a simple calculation, compare
\cite[Proposition~2.3]{daws3}.  For (\ref{prop:mult_char:two}) implies (\ref{prop:mult_char:three}) see
\cite[Theorem~4.10]{jnr}, or \cite[Proposition~3.1]{daws} for a quicker proof
which also establishes that $a\in \M(C_0(\G))$.
\end{proof}

\begin{rem}\label{rem:mult_leaves_co_inv}
Condition (\ref{prop:mult_char:three}) actually implies that $L$ restricts to a map on $C_0(\hh\QG)$.  Indeed, for $\omega\in\mc B(L^2(\QG))_*$, we see that $L\big( (\id\otimes\omega)(\widehat W) \big) = (\id\otimes\omega a) (\widehat W)$.  As $C_0(\hh\QG)$ is the closure of $\{ (\id\otimes\omega)(\widehat W) : \omega\in\mc B(L^2(\QG))_* \}$ the result follows.
\end{rem}

As explained in \cite{daws}, there is a standard way to use a representation
$U\in \M(C_0(\G)\otimes\mc K(\sH))$ of $\G$ and a bounded functional on $\mc K(\sH)$ to induce a completely bounded left multiplier $L$ of $\hh\G$.  For $\omega\in\mc K(\sH)^*$, the map defined by
\begin{equation} L(x) = (\id\otimes\omega)\big( U(x\otimes 1)U^* \big),
\qquad (x\in L^\infty(\hh\G)) \label{LfromU}\end{equation}
is a normal completely bounded map (completely positive if $\omega$
is positive) whose pre-adjoint is a left multiplier.
The associated ``representing'' element $a\in \M (C_0(\G))$ is given by
$a=(\id\otimes\omega)(U^*)$.  Recall that $b=(\id\otimes\omega)(U)\in \overline{{\mc D}}_S$ and satisfies $S(b)=a$.
If $\omega$ is self-adjoint, then $b^*=(\id\otimes\omega)(U^*) = a$.

Since every representation $U$ is of the form $U=(\id\otimes\phi_U)(\wW)$ for some non-degenerate $*$-homomorphism $\phi_U:C_0^u(\hh\G)\rightarrow\mc B(H)$, we can write $a=(\id\otimes\mu)(\wW^*) = (\mu^*\otimes\id)(\hh\Wuniv)^*$, where $\mu = \omega\circ\phi_U \in C_0^u(\hh\G)^*$.  In this way,
$L$ is of the form
\begin{equation}
L(x) = (\id\otimes\mu)(\wW(x\otimes 1)\wW^*)
= (\mu\otimes\id)(\hh\Wuniv^*(1\otimes x)\hh\Wuniv).
\end{equation}
The converse holds when $L$ is completely positive:

\begin{tw}[{\cite[Theorem 5.1]{daws}}]
Let $L:L^\infty(\hh\G)\rightarrow L^\infty(\hh\G)$ be a completely positive
map which is the adjoint of a left multiplier of $L^1(\hh\G)$.  Then
$L(x) = (\mu\otimes\id)(\hh\Wuniv^*(1\otimes x)\hh\Wuniv)$
for some positive $\mu\in C_0^u(\hh\G)^*$.
\end{tw}

In the completely positive case, one might call the resulting representing elements
$a$ ``completely positive definite''; a somewhat intrinsic characterisation
of such $a$ is given in \cite{ds}.

We will now show that completely positive multipliers induce bounded maps on the Hilbert space $L^2(\QG)$, which have a natural interpretation in terms of $U$ and $\omega$. This result, or rather its special case recorded in Proposition \ref{states-multipliers}, will be of use in the proof of Theorem \ref{groupalg}. As we believe it is of independent interest, we present the general statement.

Let us first recall certain facts related to weights.  For a proper weight $\gamma$ on a $C^*$-algebra $\alg$ we put $\mf n_{\gamma} = \{ x\in \alg :
\gamma(x^*x)<\infty\}$.  Then we have the GNS construction $(\Hil,\pi,\eta)$
where $\Hil$ is the completion of $\mf n_{\gamma}$ for the pre-inner-product
$(\eta(x)|\eta(y)) = \gamma(y^*x)$ for $x,y\in\mf n_{\gamma}$.
Let $L:\alg\rightarrow \alg$ be a positive linear map which satisfies the
Schwarz inequality $L(x)^*L(x) \leq L(x^*x)$, as would be true
for any completely positive $L$.
Assume furthermore that $\gamma(L(x)) \leq \gamma(x)$ for all $x\in \alg^+$
(under the usual convention that $t\leq\infty$ for all $0\leq t\leq\infty$).  Then there is a contractive linear map $T:\Hil\rightarrow \Hil$ which satisfies $T\eta(x) = \eta(L(x))$ for all $x\in\mf n_{\gamma}$.

\begin{tw}\label{states-multipliers-new}
Let $\G$ be a locally compact quantum group, let $U$ be a representation of
$\G$ on $\Hil$, let $\omega\in\mc K(\Hil)^*$ be a state, and form $L$ as in
\eqref{LfromU}.  Then $L$ is unital, and satisfies $\hh\varphi\circ L
\leq \hh\varphi$. Thus $L$ induces a contractive operator  $T:L^2(\hh\G) \rightarrow L^2(\hh\G)$.
Identifying $L^2(\hh\G)$ with $L^2(\G)$, the operator $T$ is equal to
$(\id\otimes\omega)(U)$.
\end{tw}
\begin{proof}
Unitality of $L$ follows directly from \eqref{LfromU} and the fact that $\omega$ is a state.  Recall the reducing morphism $\hh\Lambda:C_0^u(\hh\G) \rightarrow C_0(\hh\G)$ be the reducing morphism.  Then the composition of $L^1(\hh\G) \rightarrow C_0(\hh\G)^*$ with $\hh\Lambda^*:C_0(\hh\G)^* \rightarrow C_0^u(\hh\G)^*$ identifies $L^1(\hh\G)$ with a two-sided ideal in $C_0^u(\hh\G)^*$, see \cite[Proposition~8.3]{daws2}.  We follow \cite[Section~4.1]{daws} to see that under this identification there is a state $\mu\in C_0^u(\hh\G)^*$ such that $\hh\Lambda^*(L_*(\omega)) = \mu\hh\Lambda^*(\omega)$ for all $\omega\in L^1(\hh\G)$.  Define $L_\mu:C_0^u(\hh\G)\rightarrow C_0^u(\hh\G); x\mapsto (\mu\otimes\id)\hh\Delta_u(x)$, so that $\hh\Lambda^* \circ L_* = L_\mu^*\circ \hh\Lambda^*$.

We now follow \cite[Section~8]{ku}, and define $\hh\varphi_u = \hh\varphi\circ \hh\Lambda$, which is a proper, left-invariant weight on $C_0^u(\hh\G)$.  Then
\cite[Proposition~8.4]{ku} shows that for $\omega\in C_0^u(\hh\G)^*_+$ and $x\in C_0^u(\hh\G)^+$ with $\hh\varphi_u(x)<\infty$ we have that $\hh\varphi_u( (\omega\otimes\id)\hh\Delta_u(x) ) = \ip{\omega}{1}\hh\varphi_u(x)$.  It follows that $\hh\varphi_u(L_\mu(x)) = \hh\varphi_u(x)$ for such $x$.  So for $y\in C_0(\G)^+$ with $\hh\varphi(y)<\infty$, there is $x\in C_0^u(\G)^+$ with $\Lambda(x)=y$, so $\hh\varphi_u(x)<\infty$, hence $\hh\varphi(y) = \hh\varphi_u(x) = \hh\varphi_u(L_\mu(x)) = \hh\varphi(L(y))$.  In particular, $\hh\varphi\circ L \leq \hh\varphi$.

If we denote the GNS inclusion for the dual weight by $\hh \eta:\mf n_{\hh\varphi} \to L^2(\G)$, there is hence a contraction $T\in\mc B(L^2(\hh\G))$ with $T \hh\eta(x) = \hh\eta\big( L(x) \big)$ for $x\in\mf n_{\hh\varphi}$. Let us now recall the dual weight construction (see \cite[Section~1.1]{kvvn} or
\cite[Section~8]{KV}).  We define
\begin{equation} \mc I = \{ \gamma\in L^1(\G) : \exists\, \xi(\gamma)\in L^2(\G),
(\xi(\gamma)|\eta(a)) = \ip{a^*}{\gamma} \ (a\in\mf n_\varphi) \}, \end{equation}
where this time $\eta:\mf n_{\varphi} \to L^2(\G)$ is the GNS inclusion for $\varphi$.
If we let $\lambda:L^1(\G) \rightarrow L^\infty(\hh\G); \gamma \mapsto
(\gamma\otimes\id)(W)$ be the left-regular representation, then
$\widehat\eta(\lambda(\gamma)) = \xi(\gamma)$, under the identification of
$L^2(\hh\G)$ with $L^2(\G)$.  Moreover, for $x\in MC_0(\G)$ and $\gamma\in L^1(\G)$ we have that $\xi(x\gamma) = x \xi(\gamma)$.

As in the lines after \eqref{LfromU}, set $a = (\id\otimes\omega)(U^*)$.  Arguing as in Remark~\ref{rem:mult_leaves_co_inv}, we find that $L(\lambda(\gamma)) = \lambda(a^*\gamma)$ for $\gamma\in\mc I$.  Thus
\begin{equation} T \xi(\gamma) = T\widehat\eta(\lambda(\gamma))
= \widehat\eta(L(\lambda(\gamma)) = \widehat\eta(\lambda(a^*\gamma))
= \xi(a^*\gamma) = a^* \xi(\gamma). \end{equation}
However, $a^*=(\id\otimes\omega)(U^*)^* = (\id\otimes\omega)(U)$
as $\omega$ is positive, which completes the proof.
\end{proof}

\begin{rem}
Working a bit harder, and following \cite[Section~1.1]{KV} for how the weight on $C_0(\G)$ extends to $L^\infty(\G)$, one can show that $\varphi\circ L \leq \varphi$ on $L^\infty(\G)^+$ as well.
\end{rem}

If $\QG$ is discrete, we can actually show that the left multiplier $L$ preserves the Haar state of $\widehat{\QG}$. We formulate this below, and leave the easy proof of the state preservation to the reader.

\begin{prop}\label{states-multipliers}
Let $\G$ be a discrete quantum group, and denote the Haar state on $\hh \QG$ by  $\hh\varphi$.
Let $U$ be a representation of $\G$ on $\Hil$, let $\omega\in\mc K(\sH)^*$ be a state,
and form $L$ as in \eqref{LfromU}.  Then $L$ is unital, leaves $\hh\varphi$ invariant,
and so induces $T:L^2(\hh\G) \rightarrow L^2(\hh\G)$.  Identifying
$L^2(\hh\G)$ with $L^2(\G)$, the operator $T$ is equal to
$(\id\otimes\omega)(U)\in \M (c_0(\G))$.
\end{prop}

\section{Containment and weak containment of representations of locally compact quantum groups}

In this section we recall the notions of containment and weak containment for unitary representations of locally compact quantum groups.  Similar considerations can be found in the articles \cite{BCT} and \cite{ks}. Weak containment is defined in terms of the corresponding concept for representations of $C^*$-algebras, and so we begin by recalling this definition from \cite[Section~3.4]{dix}. A positive functional associated to a representation $\phi:\alg\rightarrow\mc B(\Hil)$ is one of the form $\omega_{x,x}\circ\phi$ for some $x\in\Hil$.

\begin{tw}\cite[Theorem~1.2]{fell}\label{thm:weakcon}
Let $\alg$ be a $C^*$-algebra, $\phi:\alg\rightarrow\mc B(\Hil)$ a representation,
and let $\mc S$ be a collection of representations of $\alg$.  The following
are equivalent, and define what it means for $\phi$ to be \emph{weakly-contained}
in $\mc S$, written $\phi \wkcon \mc S$:
\begin{enumerate}
\item $\ker\phi$ contains $\bigcap_{\pi\in\mc S} \ker\pi$;
\item every positive functional on $\alg$ associated to $\phi$ is the weak$^*$-limit
of linear combinations of positive functionals associated to representations in $\mc S$;
\item every positive functional on $A$ associated to $\phi$ is the weak$^*$-limit
of sums of positive functionals associated to representations in $\mc S$;
\item\label{thm:weakcon:four}
every positive functional $\omega$ associated to $\phi$ is the weak$^*$-limit
of sums of positive functionals associated to representations in $\mc S$ of norm at most $\|\omega\|$.
\end{enumerate}
\end{tw}

If $\phi$ is in addition
irreducible, we can avoid linear combinations, as one can show, adapting the argument from
\cite[Appendix~F]{bdv}.

For the rest of this section fix a locally compact quantum group $\QG$.  Using the bijection between unitary representations $U$ of $\QG$ on $\sH$ and $*$-homomorphisms $\phi_U:C_0^u(\hh\QG) \rightarrow \mc B(\sH)$ given by
$U = (\id\otimes\phi_U)(\wW)$,  one can define containment, weak-containment, equivalence, and weak-equivalence for unitary representations by importing the definitions for $\phi_U$, as in \cite[Section~2.3]{ks} (see also \cite[Section~5]{BT}).
\begin{deft}
Let $U,V$ be unitary representations of $\QG$ on respective Hilbert spaces $\sH_U$ and $\sH_V$, with respective $^*$-homomorphisms $\phi_U$ and $\phi_V$. Then
\begin{itemize}
\item $U$ is \emph{contained} in $V$ (that is,
$U$ is a sub-representation of $V$), which we denote by $U \leq V$,  if $\phi_U$ is contained in
$\phi_V$.  This means that there is an isometry $u:\sH_U\rightarrow \sH_V$
with $\phi_V(a) u = u\phi_U(a)$ for all $a\in C_0^u(\hh\QG)$.
Equivalently, $V (1\otimes u) = (1\otimes u) U$.
\item $U$ and $V$ are \emph{(unitarily) equivalent} if $\phi_U$ and $\phi_V$ are  equivalent, i.e.\ there is a unitary
$u:\sH_U \rightarrow \sH_V$ with $u \phi_U(a) = \phi_V(a) u$ for all
$a\in C_0^u(\hh\QG)$.
\item $U$ is \emph{weakly-contained} in $V$, which we denote by $U \wkcon V$ if
$\phi_U \wkcon \phi_V$.
\item $U$ and $V$ are \emph{weakly-equivalent} if both
$U\wkcon V$ and $V\wkcon U$.
\end{itemize}
\end{deft}

\begin{deft} Let $U\in \M(C_0(\QG) \ot \mc K(\Hil))$  be a representation of a locally compact quantum group $\QG$. A vector $\xi \in \Hil$ is said to be invariant for $U$
if $U(\eta\otimes\xi)=\eta\otimes\xi$ for all $\eta\in L^2(\QG)$. We  say that $U$ has \emph{almost invariant vectors} if there exists a net $(\xi_\alpha)$
of unit vectors in $\Hil$ such that $\| U(\eta\otimes\xi_\alpha) -
\eta\otimes\xi_\alpha \| \rightarrow 0$ for each $\eta\in L^2(\QG)$.
\end{deft}

The following proposition and corollary collects standard reformulations of containment of representations; compare with \cite[Proposition 5.1]{BT}.
\begin{prop}\label{prop:wheninv}
Let $U\in \M(C_0(\QG)\otimes\mc K(\Hil))$ be a unitary representation of $\QG$, with a corresponding adjointable operator
 $\mc U\in\mc L(C_0(\QG)\otimes\sH)$, and associated $C^*$-algebraic representation
$\phi_U:C_0^u(\hh\QG) \rightarrow \mc B(\sH)$.  Let $\xi\in\sH$.  Then the following
are equivalent:
\begin{enumerate}
\item $\xi$ is invariant for $U$;
\item $(\id\otimes\omega_{\xi,\eta})(U) = (\xi|\eta) 1$ for all $\eta\in\sH$;
\item $(\omega\otimes\id)(U) \xi = \ip{1}{\omega} \xi$ for all
  $\omega\in L^1(\G)$;
\item $\mc U(a\otimes\xi) = a\otimes\xi$ for all $a\in C_0(\QG)$.
\item $\phi_U(a)\xi = \hh\epsilon_u(a) \xi$ for all $a\in C_0^u(\hh\QG)$,
where $\hh\epsilon_u$ is the counit of $\hh \QG$.
\end{enumerate}
\end{prop}

\begin{cor}
A representation $U$ of a locally compact quantum group  has a non-zero invariant vector if and only if
$1\leq U$.
\end{cor}

We now turn to characterisations of representations $U$ which weakly contain the trivial representation. We begin with a  preparatory technical lemma.

\begin{lem}\label{lem:compact}
For any $a\in C_0(\G)$ and $\omega_0\in L^1(\G)$, the set
$\{ (a\omega a^*)*\omega_0 : \omega\in L^1(\G), \|\omega\|\leq 1 \}$
is relatively compact in $L^1(\G)$.
\end{lem}
\begin{proof}
Throughout the proof we use the fact that $L^{\infty}(\G)$ is in the standard position in $\mathcal B(L^2(\G))$, so that in particular any element $\omega\in L^1(\G)$ can be represented in the form $\omega_{\xi,\eta}$ for $\xi,\eta\in L^2(\G)$ with $\|\xi\| \|\eta\| = \|\omega\|$.
Let $\omega_0=\omega_{\xi_0,\eta_0}$ for some $\xi_0,\eta_0\in L^2(\QG)$.
Choose compact operators $\theta_1,\theta_2\in\mc K(L^2(\QG))$ with
$\theta_1(\xi_0)=\xi_0$ and $\theta_2(\eta_0)=\eta_0$.  Let $\epsilon>0$, and recalling that $W\in M(C_0(\QG)\otimes \mc K(L^2(\QG)))$, we can find linear combinations of elementary tensors
$\sum_{n=1}^N a^{(1)}_n \otimes \theta^{(1)}_n, \sum_{n=1}^N a^{(2)}_n \otimes \theta^{(2)}_n \in C_0(\QG)\otimes \mc K(L^2(\QG))$ with
\begin{equation}
\Big\| W(a\otimes\theta_1) - \sum_{n=1}^N a^{(1)}_n \otimes \theta^{(1)}_n
\Big\| < \epsilon, \quad
\Big\| W(a\otimes\theta_2) - \sum_{n=1}^N a^{(2)}_n \otimes \theta^{(2)}_n
\Big\| < \epsilon.
\end{equation}
For $\omega=\omega_{\alpha,\beta}$ for some $\alpha,\beta\in L^2(\QG)$ with $\|\alpha\|,\|\beta\|\leq 1$, then for $x\in L^\infty(\G)$,
\begin{align*} \ip{x}{(a\omega a^*) * \omega_0}
&= \ip{(a^*\otimes 1) \Delta(x)(a\otimes 1)}{\omega \otimes \omega_0} \\
&= \big( (a^*\otimes 1) W^*(1\otimes x)W(a\otimes 1)(\alpha\otimes\xi_0)
\big| \beta\otimes\eta_0 \big) \\
&= \big( (1\otimes x)W(a\otimes \theta_1)(\alpha\otimes\xi_0)
\big| W(a\otimes\theta_2)(\beta\otimes\eta_0) \big).
\end{align*}
It follows that \begin{equation}
\left|\ip{x}{(a\omega a^*) * \omega_0}-\sum_{n,m=1}^N ( a^{(1)}_n \alpha | a^{(2)}_m \beta )
( x \theta^{(1)}_n \xi_0 | \theta^{(2)}_n \eta_0 )\right|<2\epsilon\|x\|,\qquad (x\in L^\infty(\QG)),
\end{equation}
and hence $\|(a\omega a^*)*\omega_0- \sum_{n,m=1}^N \ip{ (a^{(2)}_m)^* a^{(1)}_n }{\omega} \omega_{ \theta^{(1)}_n \xi_0 , \theta^{(2)}_n \eta_0 }\|\leq 2\epsilon$ for all $\|\omega\|\leq 1$.  Now, the set
\begin{equation}
\left\{\sum_{n,m=1}^N\ip{ (a^{(2)}_m)^* a^{(1)}_n }{\omega}
\omega_{ \theta^{(1)}_n \xi_0 , \theta^{(2)}_n \eta_0 }:\omega\in L^1(\QG),\ \|\omega\|\leq1\right\}
\end{equation}
is clearly compact and so has a finite $\epsilon$-net, which hence forms a finite $3\epsilon$-net for $\{ (a\omega a^*)*\omega_0 : \omega\in L^1(\G), \|\omega\|\leq 1 \}$.
\end{proof}

Next we record characterisations of those representations $U$ which weakly contain the trivial representation. As $\hh\epsilon_u$ is irreducible, these are precisely those for which $\hh\epsilon_u$ is the weak$^*$-limit of states
of the form $\omega_{\xi}\circ\phi_U$; this is condition
(\ref{prop:inv_vects:four}) in the following proposition (for an alternative
approach, see \cite[Theorem~5.1]{BT}).  In \cite{BCT} and \cite[Section~5]{BT} the terminology $U$ has WCP (the weak containment property) is used for those representations $U$ with $\hh\epsilon_u\wkcon \phi_U$ --- here we use the terminology that $U$ has almost invariant vectors for this condition.

\begin{prop}\label{prop:inv_vects}
Let $U\in \M(C_0(\QG)\otimes\mc \mc K(\Hil))$ be a unitary representation of $\QG$, with a corresponding adjointable operator
 $\mc U\in\mc L(C_0(\QG)\otimes\sH)$, and associated $C^*$-algebraic representation
$\phi_U:C_0^u(\hh\QG) \rightarrow \mc B(\sH)$.  Let $(\xi_\alpha)$ be a net
of unit vectors in $\Hil$.  The following are equivalent:
\begin{enumerate}
\item\label{prop:inv_vects:one}
$\| U(\eta\otimes\xi_\alpha) -
\eta\otimes\xi_\alpha \| \rightarrow 0$ for each $\eta\in L^2(\QG)$;
\item\label{prop:inv_vects:six}
The net $(\id\otimes\omega_{\xi_\alpha})(U)$ converges weak$^*$ to $1$
in $L^\infty(\QG)$;
\item\label{prop:inv_vects:four}
The net of states $(\omega_{\xi_\alpha} \circ \phi_U)_\alpha$ on $C_0^u(\hh \QG)$  converges weak$^*$
to $\hh\epsilon_u$ in $C_0^u(\hh\QG)^*$;
\item\label{prop:inv_vects:two}
$\|\phi_U(a)\xi_\alpha - \hh\epsilon_u(a) \xi_\alpha\| \rightarrow 0$
for all $a\in C_0^u(\hh\QG)$;
\item\label{prop:inv_vects:five}
$\| \mc U(a\otimes\xi_\alpha) - a\otimes\xi_\alpha \| \rightarrow 0$
for all $a\in C_0(\QG)$;
\end{enumerate}
Moreover the existence of a net of unit vectors satisfying the equivalent conditions above is equivalent to the following statement:
\begin{enumerate}
\setcounter{enumi}{5}
\item\label{prop:inv_vects:three}
there is a state $\mu_0\in \mc B(\sH)^*$ such that $(\id\otimes\mu_0)(U)
= 1 \in L^\infty(\QG)$.
\end{enumerate}
\end{prop}
\begin{proof}

The equivalences (1)$\Longleftrightarrow$(2)$\Longleftrightarrow$(3)$\Longleftrightarrow$(4)$\Longleftrightarrow$(6) are either straightforward or can be viewed as variants of the results obtained
in the articles \cite{BT} and \cite{BMT}. We thus look only at (\ref{prop:inv_vects:five}).  If (\ref{prop:inv_vects:five})
holds, then
\begin{equation}
 0 = \lim_\alpha \big( \mc U(a\otimes\xi_\alpha) - a\otimes\xi_\alpha
\big| b\otimes\xi_\alpha \big)
= \lim_\alpha b^* (\id\otimes\omega_{\xi_\alpha})(U) a
- b^* a,\qquad (a,b\in C_0(\QG)),
\end{equation}
from which (\ref{prop:inv_vects:six}) follows, as $C_0(\G)$ acts non-degenerately on $L^2(\G)$.

We prove the converse using Lemma \ref{lem:compact}.  For $\omega\in L^1(\G)$, let $T_\omega
= (\omega\otimes\id)(U)$, so as $U$ is a representation, $T_{\omega'}
T_\omega = T_{\omega'*\omega}$.  Suppose that (\ref{prop:inv_vects:two}) holds,
so as $T_\omega = \phi_U((\omega\otimes\id)\wW)$, it follows that
$\|T_\omega \xi_\alpha - \ip{1}{\omega}
\xi_\alpha \| \rightarrow 0$ for all $\omega\in L^1(\G)$.
Fix $a\in C_0(\G)$ and $\xi_0\in L^2(\G)$ both of norm one, and set
$\omega_0 = \omega_{\xi_0}\in L^1(\G)$.  Now consider
\begin{align}
\big\| \mc U(a\otimes T_{\omega_0}\xi_\alpha) - a\otimes\xi_\alpha \big\|^2
&= \big\| \|T_{\omega_0}\xi_\alpha\|^2 a^*a + a^*a
   - 2\Re\big( (\mc U(a\otimes T_{\omega_0}\xi_\alpha) |
   a\otimes\xi_\alpha) \big) \big\|.
\end{align}
We then see that
\begin{align}
\big( \mc U(a\otimes T_{\omega_0}\xi_\alpha) \big| a\otimes\xi_\alpha \big)
&= a^* ( \id\otimes\omega_{T_{\omega_0}\xi_\alpha,\xi_\alpha})(U) a
= a^* ( \id\otimes\omega_{\xi_\alpha})(U(1\otimes T_{\omega_0})) a \nonumber\\
&= a^* (\id\otimes\omega_0\otimes\omega_{\xi_\alpha})(U_{13} U_{23}) a
= a^* (\id\otimes\omega_0\otimes\omega_{\xi_\alpha})
   ((\Delta\otimes\id)(U)) a \nonumber\\
&= a^* (\id\otimes\omega_0)\Delta\big( (\id\otimes\omega_{\xi_\alpha})(U)
   \big) a.\label{2.14}
\end{align}
As $\omega_0$ is a state, $\lim_\alpha \|T_{\omega_0}\xi_\alpha - \xi_\alpha\| =0$,
and so
\begin{equation} \lim_\alpha \big\| \mc U(a\otimes \xi_\alpha)
   - a\otimes\xi_\alpha \big\|
= \lim_\alpha \big\| \mc U(a\otimes T_{\omega_0}\xi_\alpha)
   - a\otimes\xi_\alpha \big\|,
 \end{equation}
and by (\ref{2.14}), this limit will be zero if and only if
\begin{equation} \lim_\alpha a^* (\id\otimes\omega_0)\Delta\big(
   (\id\otimes\omega_{\xi_\alpha})(U) \big) a   =  a^*a. \label{eq:matttwo}
\end{equation}
As (\ref{prop:inv_vects:two}) holds, and hence (\ref{prop:inv_vects:six}) holds,
it follows that
\begin{align}\ip{a^*a}{\omega} &= \ip{1}{(a\omega a^*)*\omega_0}
= \lim_\alpha \ip{(\id\otimes\omega_{\xi_\alpha})(U)}
   {(a\omega a^*)*\omega_0}\nonumber\\
&= \lim_\alpha \ip{a^* (\id\otimes\omega_0)
   \Delta\big((\id\otimes\omega_{\xi_\alpha})(U) \big) a}{\omega},\label{2.17}
 \end{align}
 for each $\omega\in L^1(\G)$.  By Lemma~\ref{lem:compact} the set $\{ (a\omega a^*)*\omega_0 :
\omega\in L^1(\G), \|\omega\|\leq 1 \}$ is relatively compact in $L^1(\G)$ and hence the limit in (\ref{2.17})
holds uniformly over the set $\{ \omega\in L^1(\G): \|\omega\|\leq 1 \}$.
This implies that (\ref{eq:matttwo}) holds, as required to show that (\ref{prop:inv_vects:five}) holds.
\end{proof}

\begin{cor}\label{prop:inv_vects1}
A representation $U$ of a locally compact quantum group  admits almost invariant vectors if and only if
$1\wkcon U$ (equivalently, $\hh\epsilon_u$ is weakly-contained in $\phi_U$).
\end{cor}

\section{Mixing representations}

In this section we introduce  mixing (or $C_0$) representations of locally compact quantum groups and analyse their properties.

\begin{deft}
 A representation $U\in \M(C_0(\QG) \ot \mc K(\Hil))$ is said  to be \emph{mixing} if it has $C_0$--coefficients, which means that for all $\xi, \eta \in \Hil$, we have $(\id \ot \omega_{\xi, \eta}) (U)\in C_0(\QG)$.
\end{deft}

The origins of the term mixing lie in the theory of dynamical systems -- an action of a group $G$ on a probability space $(X,\mu)$ is mixing in the usual dynamical sense (see Definition 3.4.6 in \cite{Glasner}) if and only if the associated Koopman-type representation of $G$ on $L^2(X,\mu)_0:=L^2(X,\mu) \ominus \bc 1$ is mixing.
\begin{prop}\label{prop:easy_prop_mix}
Let $U,V$ be representations of $\G$.  Then:
\begin{enumerate}
\item\label{prop:easy_prop_mix:one}
If $U$ and $V$ are mixing, then so is $U\oplus V$.
\item\label{prop:easy_prop_mix:two}
If $U$ is mixing, then so are $U\tp V$ and $V\tp U$.
\end{enumerate}
\end{prop}
\begin{proof}
(1) is routine. For (\ref{prop:easy_prop_mix:two}) test on elementary tensors, as $C_0(\G)$ is an ideal in $\M (C_0(\G))$.
\end{proof}

The next lemma, connecting the mixing property of a representation to the properties of a certain state, will be used in Section 5.

\begin{lem} \label{mixequiv}
Let $\mu$ be a state on $C_0^u(\hh\G)$, and let $x=(\id\otimes\mu)(\wW)
\in \M(C_0(\G))$.  Let $(\phi,\sH,\xi)$ be the GNS construction for $\mu$,
and let $U$ be the representation of $\G$ associated to $\phi:C_0^u(\hh\G) \to \mc B(\Hil)$.  Then $U$ is
mixing if and only if $x\in C_0(\G)$.
\end{lem}
\begin{proof}
We have that $U = (\id\otimes\phi)(\wW)$.  If $U$ is mixing,
then $x = (\id\otimes\omega_{\xi,\xi}\circ\phi)(\wW)
= (\id\otimes\omega_{\xi,\xi})(U) \in C_0(\G)$.

Conversely, let $a,b\in C_0^u(\hh\G)$ and set $\alpha=\phi(a)\xi$ and $\beta=\phi(b)\xi$.  Suppose further that $a=(\omega_1\otimes\id)(\wW)$ and $b^*=(\omega_2\otimes\id)(\wW)$ for some $\omega_1,\omega_2\in L^1(\QG)$.
As $\wW$ is a representation of $\G$,
\begin{align*}
(1\otimes b^*)\wW(1\otimes a)
&= (\omega_1\otimes\omega_2\otimes\id\otimes\id)\big( \wW_{24}
   \wW_{34} \wW_{14} \big) \nonumber\\
&= (\omega_1\otimes\omega_2\otimes\id\otimes\id)\big(
   (1\otimes(\Delta\otimes\id)(\wW)) \wW_{14} \big) \nonumber\\
&= (\omega_1\otimes\omega_2\otimes\id\otimes\id)\big(
   (\id\otimes\Delta\otimes\id)(\wW_{23}\wW_{13}) \big) \nonumber\\
&= (\omega_1\otimes\omega_2\otimes\id\otimes\id)\big(
   (\id\otimes\Delta\otimes\id)(\Delta^\op\otimes\id)(\wW) \big).
\end{align*}
Here $\Delta^\op = \sigma\circ\Delta$ is the opposite coproduct.  Now set
$y=(\id\otimes\omega_{\alpha,\beta})(U)$, so that
\begin{align*}
y &= (\id\otimes (\phi(a) \omega_{\xi,\xi} \phi(b)^*)\circ\phi)(\wW)
= (\id\otimes\omega_{\xi,\xi}\circ\phi)
   \big( (1\otimes b^*)\wW(1\otimes a) \big) \nonumber\\
&= (\omega_1\otimes\omega_2\otimes\id\otimes\mu)\big(
   (\id\otimes\Delta\otimes\id)(\Delta^\op\otimes\id)(\wW) \big)\nonumber \\
&= (\omega_1\otimes\omega_2\otimes\id)\big(
   (\id\otimes\Delta)\Delta^\op(x) \big)
= (\omega_2\otimes\id)\Delta( (\id\otimes\omega_1)\Delta(x) ).
\end{align*}
Now, by \cite[Corollary~6.11]{KV} we have that $\Delta(d)(c\otimes 1) \in
C_0(\G)\otimes C_0(\G)$ for any $c,d\in C_0(\G)$.  By Cohen Factorisation (see for example \cite[Appendix~A]{mnw}), any $\omega\in L^1(\G)$
has the form $\omega = c\omega'$ for some $c\in C_0(\G), \omega'\in L^1(\G)$.
Thus $(\omega\otimes\id)\Delta(x) = (\omega'\otimes\id)\big( \Delta(x)(c\otimes 1) \big) \in C_0(\G)$.  Similarly, we can show that $(\id\otimes\omega)\Delta(d)\in C_0(\G)$ for any $\omega\in L^1(\G), d\in C_0(\G)$.
It follows that $y = (\omega_2\otimes\id)\Delta( (\id\otimes\omega_1)\Delta(x) ) \in C_0(\G)$.  As $a,b$ as above are dense in $C_0^u(\hh\G)$, and thus $\alpha,\beta$ as above are dense in $\sH$, we have shown that $U$ has $C_0$-coefficients, that is, $U$ is mixing.
\end{proof}

\section{Topologising representations of locally compact quantum groups}\label{sec:rep_met_sp}

Let $\QG$ denote a second countable locally compact quantum group.  In this section, we equip the set of all unitary representations of $\QG$ on a fixed infinite-dimensional, separable Hilbert space $\Hil$ with a natural Polish topology (for the analogous concepts in the classical framework, see the book \cite{Kechris}) and give conditions for density of classes of representations.

Fix  an infinite-dimensional separable Hilbert space $\sH$, and a unitary $u:\sH \rightarrow
\sH\otimes\sH$.  Let $\rep_\G(\sH)$ denote the collection of unitary representations
of $\G$ on $\sH$.  This is a monoidal category for the product
\begin{equation} U \boxtimes V = (1\otimes u^*)(U\tp V)(1\otimes u).
\end{equation}
Note the use of $u$ in this definition ensures that $U\boxtimes V$ is a representation on
$\sH$ and not $\sH\otimes\sH$.

When $\alg$ is a separable $C^*$-algebra the unitary group $\mathcal U(\M(\alg))$ of $\M(\alg)$ is Polish in the strict topology (see \cite[Page 191]{RW} for example). As multiplication is strictly continuous on bounded sets, $\rep_\G(\sH)$ is strictly closed in $\M(C_0(\G)\otimes\mc K(\Hil))$ and so is Polish in the relative strict topology.

Denote by $\rep(C_0^u(\hh\G),\sH)$ the set of non-degenerate $^*$-representations
of $C_0^u(\hh\G)$ on $\sH$.  The following proposition implies that $\rep_\G(\sH)$ is a \emph{topological $W^*$-category} in the sense of \cite{worcat}, equivalent to the $W^*$-category $\textup{Rep}(C_0^u(\hh\G,\sH))$.  Recall that on bounded sets the strong$^*$-topology on $\mathcal B(\sH)$ agrees with the strict topology (defined via $\mc B(\sH)\cong\M(\mc K(\sH))$).

\begin{prop}
Under the bijection between $\rep_\G(\sH)$ and $\rep(C_0^u(\hh\G),\sH)$, the topology induced on $\rep(C_0^u(\hh\G),\sH)$ is the point-strict topology (so $\phi_n \stackrel{n\to \infty}{\longrightarrow}\phi$ if and only if, for each $\widehat a \in C_0^u(\hh\G)$, we have that $\phi_n(\widehat a)\stackrel{n\to\infty}{\longrightarrow} \phi(\widehat a)$ strictly in $\mc B(\sH)=\M(\mc K(\sH))$).
\end{prop}
\begin{proof}
Let $(U_n)_{n=1}^{\infty}$ be a sequence in $\rep_\G(\sH)$ with the corresponding sequence
$(\phi_n)_{n=1}^{\infty}$ in $\rep(C_0^u(\hh\G),\sH)$; similarly let $U\in \rep_\G(\sH)$ and $\phi\in \rep(C_0^u(\hh\G),\sH)$ correspond.
Firstly, suppose that $\phi_n\stackrel{n\to \infty}{\longrightarrow}\phi$ in the point-strict topology.
Let $a\in C_0(\G),\widehat a\in C_0^u(\hh\G)$ and $\theta\in\mc K(\Hil)$, so that
\begin{align} U_n(a\otimes \phi_n(\widehat a)\theta) &=
(\id\otimes\phi_n)(\wW)(a\otimes\phi_n(\widehat a)\theta)
= (\id\otimes\phi_n)\big( \wW(a\otimes\widehat a) \big) (1\otimes\theta).
\end{align}
As $\wW \in \M(C_0(\G)\otimes C_0^u(\hh\G))$ it follows that
$\wW(a\otimes\widehat a) \in C_0(\G)\otimes C_0^u(\hh\G)$ and so
\begin{equation} \lim_{n\to \infty} \
(\id\otimes\phi_n)\big( \wW(a\otimes\widehat a) \big) (1\otimes\theta)
= (\id\otimes\phi)\big( \wW(a\otimes\widehat a) \big) (1\otimes\theta)
= U(a\otimes \phi(\widehat a)\theta). \end{equation}
Finally observe that $\phi_n(\widehat a)\theta \stackrel{n\to\infty}{\longrightarrow} \phi(\widehat a)\theta$
in norm, and so we may conclude that $U_n(a\otimes\phi(\widehat a)\theta)
\stackrel{n\to\infty}{\longrightarrow} U(a\otimes \phi(\widehat a)\theta)$.  Similarly, we can show that
$(a\otimes\phi(\widehat a)\theta)U_n \stackrel{n\to\infty}{\longrightarrow} (a\otimes \phi(\widehat a)\theta)U$.
As $\phi$ is non-degenerate, the collection of such $\phi(\widehat a)\theta$ forms
a linearly dense subspace of $\mc K(\Hil)$, and it follows that
$U_n\stackrel{n\to\infty}{\longrightarrow} U$ strictly, as required.

Conversely, suppose that $U_n\stackrel{n\to\infty}{\longrightarrow} U$ strictly.  Let $a\in C_0(\G),\omega
\in L^1(\G)$, and set $\widehat a = (a\omega\otimes\id)(\wW) \in C_0^u(\hh\G)$.
For $\theta\in\mc K(\Hil)$,
\begin{equation}\phi_n(\widehat a)\theta = (a\omega\otimes\id)(U_n)\theta
= (\omega\otimes\id)\big( U_n(a\otimes\theta) \big)
\stackrel{n\to\infty}{\longrightarrow} (\omega\otimes\id)\big( U(a\otimes\theta) \big)
= \phi(\widehat a)\theta.
\end{equation}
By Cohen-Factorisation, we can find $\omega',a'$ with $a\omega=\omega' a'$,
and so by repeating the argument on the other side, it follows that
$\phi_n(\widehat a) \stackrel{n\to\infty}{\longrightarrow}\phi(\widehat a)$ strictly.  As elements $\widehat a$ arising in this way are
dense in $C_0^u(\hh\G)$, it follows that $\phi_n\stackrel{n\to\infty}{\longrightarrow}\phi$ in the
point-strict topology, as required.
\end{proof}

As the multiplicity of representations does not play a role when weak containment is considered, and we want to consider the trivial representation as an element of $\rep_\G(\sH)$, we will use the notation $1$ now for the unitary representation $U=1\ot 1 \in \M(C_0(\G)\otimes\mc K(\Hil))$.

\begin{prop} \label{mix-dense imply HAP}
If the mixing representations are dense in $\rep_\G(\sH)$, then there is a mixing
representation $U \in \rep_\G(\sH)$ with $1 \wkcon U$ (that is, $U$ has almost invariant
vectors).
\end{prop}
\begin{proof}
 By assumption, there is a sequence
$(U_n)_{n=1}^{\infty}$ of mixing representations such that $U_n\stackrel{n\to\infty}{\longrightarrow} 1$.  Fix a unit vector $\xi_0\in\sH$. Consider $U = \bigoplus_{n\in \bn} U_n$, which is a mixing representation on $\bigoplus_{n\in\bn}\sH\cong\sH\otimes\ell^2(\bn)$ (as coefficients of
$U$ will be norm limits of sums of coefficients of the representations $U_n$, and so will still be members of $C_0(\G)$).

Fix a unitary
$v:\sH \rightarrow \sH \otimes \ell^2(\mathbb N)$, and define
\begin{equation} U = \sum_{n\in \bn} (1\otimes v^*)\big( U_n \otimes \theta_{\delta_n,\delta_n} \big)
(1\otimes v)\in\rep_\G(\sH),
\end{equation}
where $\theta_{\delta_n,\delta_n}$ is the rank-one orthogonal
projection onto the span of $\delta_n\in\ell^2(\bn)$. Let $\xi_n = v^*(\xi_0\otimes\delta_n)$ for each $n\in \bn$.  Then, for
$\eta\in L^2(\G)$,
\begin{align} \| U(\eta\otimes\xi_n) - \eta\otimes\xi_n \|
&= \| (1\otimes v^*)(U_n(\eta\otimes\xi_0) \otimes \delta_n)
- (1\otimes v^*)(\eta\otimes\xi_0 \otimes \delta_n) \| \nonumber\\
&= \| U_n(\eta\otimes\xi_0) - \eta\otimes\xi_0 \|.
\end{align}
Now, strict convergence in $\M(C_0(\G)\otimes\mc K(\sH))$ implies
strong convergence in $\mc B(L^2(\G)\otimes\sH)$, and so
$U_n(\eta\otimes\xi_0)$ converges in norm to $\eta\otimes\xi_0$.
It follows that we have verified condition (1) of Proposition~\ref{prop:inv_vects}
for the sequence $(\xi_n)$.  Hence $U$ has almost invariant vectors.
\end{proof}

The following lemma abstracts calculations used in the classical situation for establishing density of mixing representations in \cite{BR} and weak mixing representations in \cite{KerrPichot}.

\begin{lem}\label{lem:den_cond}
Let $\mc R \subseteq \rep_\G(\sH)$ be a collection which is:
\begin{enumerate}
\item\label{lem:den_cond:one} \emph{stable under unitary equivalence}: i.e.
for a unitary $v$ on $\sH$, we have that $(1\otimes v^*)U(1\otimes v)
\in \mc R$ if and only if $U\in\mc R$;
\item\label{lem:den_cond:two} \emph{stable under tensoring with another representation}: i.e.
if $U\in\mc R, V\in\rep_\G(\sH)$ then $U\boxtimes V\in\mc R$;
\item\label{lem:den_cond:three}\emph{contains a representation with almost invariant vectors}: i.e. by Corollary \ref{prop:inv_vects1} there is $U^{(0)}\in\mc R$ with $1\wkcon U^{(0)}$.
\end{enumerate}
Then $\mc R$ is dense in $\rep_\G(\sH)$.
\end{lem}
\begin{proof}
Assume that $\mc R$ is such a collection and fix $U^{(0)}\in\mc R$ with $1\wkcon U^{(0)}$.  We will use the isomorphism $\M(C_0(\G)\otimes\mc K(\Hil))
\cong \mc L(C_0(\G)\otimes\sH)$, see (\ref{eq:mattthree}).  As $U$ has almost invariant vectors, we can find a net $(\xi_\alpha)$ of unit vectors in $\Hil$  with
$\| \mc U^{(0)}(a\otimes\xi_\alpha) - a\otimes\xi_\alpha \|\rightarrow 0$
for $a\in C_0(\G)$.

Fix $V\in\rep_\G(\sH)$. We will show that $V$ can be approximated by a sequence of elements of $\mc R$.
Let $a\in C_0(\G)$, let $\eta\in\sH$ be a unit vector, and
let $\epsilon>0$.  Let $(e_n)_{n=1}^{\infty}$ be an orthonormal basis for $\sH$,
and let
\begin{equation}\mc V(a\otimes\eta) = \sum_{n=1}^{\infty} x_{n} \otimes e_n,
\qquad \mc V^*(a\otimes\eta) = \sum_{n=1}^{\infty} y_{n} \otimes e_n.
\end{equation}
for some $x_n,y_n\in C_0(\G)$, with convergence in the Hilbert module $C_0(\G)\otimes\sH$.  Choose $N\in \bn$ so that
\begin{equation} \Big\| \sum_{n>N} x_{n} \otimes e_n \Big\|
= \Big\| \sum_{n>N} x_{n}^*x_{n} \Big\|^{1/2} < \epsilon/3,
\qquad \Big\| \sum_{n>N} y_{n} \otimes e_n \Big\| < \epsilon/3.
\end{equation}
Further choose $\alpha$ so that for all $n\leq N$
\begin{equation} \big\| \mc U^{(0)}(x_{n}\otimes\xi_\alpha) -
x_{n}\otimes\xi_\alpha \big\| < \epsilon/3N, \qquad
\big\| \mc U^{(0)}(a\otimes\xi_\alpha) - a\otimes\xi_\alpha \big\|
< \epsilon/3.
\end{equation}
Finally, set $X = \{ \eta\} \cup \{ e_n : n\leq N \}$,
a finite subset of $\sH$.  As $\sH$ is infinite-dimensional, we can
find a unitary $v:\sH\rightarrow\sH\otimes\sH$ with $v(\xi) = \xi_\alpha \otimes \xi$ for all $\xi\in X$.  Then
\begin{align}
\big\| \big( (1\otimes v^*) (\mc U^{(0)} \tp \mc V) & (1\otimes v) -
  \mc V \big) (a\otimes\eta) \big\|
= \big\| \mc U^{(0)}_{12} \mc V_{13} (a\otimes\xi_\alpha\otimes\eta)
- (1\otimes v)\mc V(a\otimes\eta) \big\| \nonumber\\
&= \big\| \mc U^{(0)}_{12} \sum_{n=1}^{\infty} x_{n} \otimes \xi_\alpha\otimes e_n
- \sum_{n=1}^{\infty} x_{n} \otimes v(e_n) \big\| \nonumber\\
&\leq 2\epsilon/3 +
\big\| \mc U^{(0)}_{12} \sum_{n\leq N} x_{n} \otimes \xi_\alpha\otimes e_n
- \sum_{n\leq N} x_{n} \otimes v(e_n) \big\|\nonumber \\
&= 2\epsilon/3 +
\big\| \sum_{n\leq N} \big(\mc U^{(0)}(x_{n} \otimes \xi_\alpha)
- x_{n}\otimes\xi_\alpha \big) \otimes e_n \big\|\nonumber \\
&\leq \epsilon.
\end{align}
Similarly,
\begin{align}
\big\| \big( (1\otimes v^*) (\mc U^{(0)} \tp \mc V)^* & (1\otimes v) -
  \mc V^* \big) (a\otimes\eta) \big\|
= \big\| \mc V_{13}^* \mc U^{(0)}_{12}{}^* (a\otimes\xi_\alpha\otimes\eta)
  - (1\otimes v)\mc V^*(a\otimes\eta) \big\| \nonumber\\
&\leq \epsilon/3 +
  \big\| \mc V_{13}^* (a\otimes\xi_\alpha\otimes\eta)
  - (1\otimes v)\mc V^*(a\otimes\eta) \big\| \nonumber\\
&= \epsilon/3 +
  \Big\| \sum_{n\in \bn} y_{n} \otimes \xi_\alpha \otimes e_n
  - y_{n} \otimes v(e_n) \Big\| \nonumber\\
&= \epsilon/3 +
  \Big\| \sum_{n>N} y_{n} \otimes \big( \xi_\alpha \otimes e_n
  - v(e_n) \big) \Big\| \nonumber\\
&\leq \epsilon/3 +
  \Big\| \sum_{n>N} y_{n} \otimes \xi_\alpha \otimes e_n \Big\|
  + \Big\| \sum_{n>N} y_{n} \otimes v(e_n) \Big\| < \epsilon.
\end{align}

The intertwiner $v$ is not equal to our fixed intertwiner $u$, so
$(1\otimes v^*) (U^{(0)} \tp V) (1\otimes v)$ need not be
equal to $U\boxtimes V$, but it is unitarily equivalent to it.  By
(\ref{lem:den_cond:one}) and (\ref{lem:den_cond:two}), it follows that
$(1\otimes v^*) (U^{(0)} \tp V) (1\otimes v) \in \mc R$. In this way we can construct a net $(V_i)_{i\in \Ind}$ in $\mc R$ such that
\begin{equation} \mc V_i(a\otimes\eta) \stackrel{i \in \Ind}{\longrightarrow} \mc V(a\otimes\eta),
\quad \mc V_i^*(a\otimes\eta) \stackrel{i \in \Ind}{\longrightarrow} \mc V^*(a\otimes\eta)
\end{equation}
for all $a\in C_0(\G), \eta\in\Hil$.  As $(\mc V_i-\mc V)_{i\in \Ind}$ is a bounded
net and we are dealing with linear maps, this is enough to show that $\mc V_i\rightarrow\mc V$ strictly, as
required.
\end{proof}

\section{The Haagerup approximation property}\label{sec:haaprop}

In this section we introduce the notion of the Haagerup property for locally compact quantum groups and provide several equivalent characterisations.

\begin{deft}
A locally compact quantum group $\QG$ has the \emph{Haagerup property}  if there exists a mixing representation of $\QG$ which has almost invariant vectors.
\end{deft}

The following statements are an immediate consequence of the definition. Recall that a locally compact quantum group \emph{has Property (T)} if each of its representations with almost invariant vectors has a nontrivial invariant vector -- this notion was introduced (for discrete quantum groups) in \cite{PierreT} and later studied in \cite{ks}.

\begin{prop} \label{coamenlcqg}
If $\widehat{\QG}$ is coamenable, then $\QG$ has the Haagerup property. Further, $\QG$ is compact if and only if $\QG$ has both the Haagerup property and Property (T).
\end{prop}

\begin{proof}
It is easy to see that the left regular representation of $\QG$ (given by the fundamental unitary $W\in \M(C_0(\QG) \ot C_0(\hh \QG))$) is mixing.  By
\cite[Theorem~3.1]{BT}, $W$ has almost invariant vectors if and only if $\hh\QG$ is coamenable.

If $\QG$ has both $(T)$ and the Haagerup property, then it has a mixing representation with a non-trivial invariant vector. However, then the corresponding coefficient is a non-zero scalar multiple of unit in $\M (C_0(\QG))$, which belongs to $C_0(\QG)$. Thus $\QG$ is compact.

On the other hand, if $\QG$ is compact, then each representation which has almost invariant vectors actually has invariant vectors, so $\QG$ has Property (T) (see Theorem 7.16 of \cite{BCT}).  As $\hh\QG$ is discrete, it is  coamenable (\cite[Proposition 5.1]{BMT}), so $\QG$ has the Haagerup property.
\end{proof}

\begin{rem}\label{rem:amen_not_enough}
Note that we do not know whether every amenable locally compact quantum group has the Haagerup property (although this is true for discrete quantum groups, see Proposition \ref{amenHAP} below). Formally, providing the answer to this question should be easier than deciding the equivalence of amenability of $\QG$ and coamenability of $\widehat{\QG}$ (a well-known open problem) but they appear to be closely related.
\end{rem}

The above proposition allows us to provide the first examples of non-discrete locally compact quantum groups with the Haagerup property.

\begin{exam}
The locally compact quantum groups quantum $E_\mu(2)$ (\cite{worE2}), its dual $\hat{E}_{\mu}(2)$ (\cite{worE2} and \cite{E(2)dual}),
quantum \hbox{$az+b$} (\cite{worAzb}) and  quantum \hbox{$ax+b$} (\cite{worZak})
have the Haagerup property. Indeed, they are all coamenable, as follows for example from Theorem 3.14 of \cite{SS} (that result does not mention $\hat{E}_{\mu}(2)$, but the information contained in Section 1 of \cite{E(2)dual} suffices to construct a bounded counit on the $C^*$-algebra $C_0(\hat{E}_{\mu}(2))$ and this in turn implies coamenability, see Theorem 3.1 in \cite{BT}), and the two last examples are self-dual, up to `reversing the group operation', i.e.\ flipping the legs of the coproduct  (see the original papers or \cite{PuszSoltan}).
\end{exam}

For part (\ref{equivHAP:four}) of the following, we recall from
Section~\ref{sec:mults} that if $L$ is a completely positive multiplier of
$L^1(\hh\G)$ then there is a ``representing element'' $a\in \M (C_0(\G))$
such that $a\hh\lambda(\hh\omega) = \hh\lambda(L(\hh\omega))$ for all
$\hh\omega\in L^1(\hh\G)$.

\begin{tw}\label{equivHAP}
Let $\QG$  be a locally compact quantum group. The following conditions are equivalent:
\begin{rlist}
\item $\QG$ has the Haagerup property;\label{equivHAP:one}
\item the mixing representations form a dense  subset of $\RepGH$;
\item\label{equivHAP:three} there exists a net of states $(\mu_i)_{i\in \mathcal{I}}$ on $C_0^u(\widehat{\QG})$ such that the net $((\id \ot \mu_i)(\wW))_{i\in \mathcal{I}}$  is an approximate identity in $C_0(\QG)$;
\item\label{equivHAP:four}
there is a net $(a_i)_{i\in \mathcal{I}}$ in $C_0(\G)$ of representing elements of
completely positive multipliers which forms an approximate identity for
$C_0(\G)$.
\end{rlist}
\end{tw}
\begin{proof}
(i)$\Longrightarrow$(ii): This follows from Lemma \ref{lem:den_cond}, as the class of mixing representations certainly satisfies conditions (\ref{lem:den_cond:one}) and (\ref{lem:den_cond:two}), the latter by Proposition~\ref{prop:easy_prop_mix}, and condition (\ref{lem:den_cond:three}) is the hypothesis of having the Haagerup property.

(ii)$\Longrightarrow$ (i): This is precisely Proposition \ref{mix-dense imply HAP}.

(iii)$\Longrightarrow$ (i): For each $i\in \Ind$ let $(\phi_i,\sH_i,\xi_i)$
be the GNS construction corresponding to $\mu_i$.  By Lemma~\ref{mixequiv}, each of the representations $U_{\phi_i}$ (associated to $\phi_i$) is mixing, and so $U:=\bigoplus_{i\in \Ind} U_{\phi_i}$ will also be mixing.   Let $x_i = (\id\otimes\mu_i)(\wW)\in C_0(\G)$, so that
$(x_i)_{i\in \mathcal{I}}$ is a bounded approximate identity for $C_0(\G)$.
Then, for $a\in C_0(\G)$,
\begin{align} \big\| \mc U(a\otimes\xi_i) - a\otimes\xi_i \big\|^2
&= \big\| 2a^*a - 2\Re\big( \mc U(a\otimes\xi_i)
   \big| a\otimes\xi_n \big) \big\|
= \big\| 2a^*a - 2\Re\big( a^* (\id\otimes\omega_{\xi_i,\xi_i})(U_i)
   a \big) \| \nonumber\\
&= \big\| 2a^*a - 2\Re\big( a^* x_i a \big) \|
= \big\| a^*(2-x_i-x_i^*)a \big\|.\label{5.5.1}
\end{align}
As $(x_i)_{i\in \mathcal{I}}$ is a bounded approximate identity, (\ref{5.5.1}) converges
to $0$ for each fixed $a\in C_0(\G)$, so by Proposition~\ref{prop:inv_vects},
$U$ has almost invariant vectors.  Thus $\G$ has the Haagerup property.

(i)$\Longrightarrow$ (iii):
As $\G$ has the Haagerup property, there exists a mixing representation $U$ of $\QG$
with almost invariant vectors, say $(\xi_i)_{i\in \mathcal{I}}$.
Let $\phi$ be the representation of $C_0^u(\widehat{\QG})$ associated to $U$, and for each $i\in \mathcal{I}$ set $\mu_i = \omega_{\xi_i,\xi_i}\circ\phi$.
As $U$ is mixing, $x_i=(\id\otimes\mu_i)(\wW) =
(\id\otimes\omega_{\xi_i,\xi_i})(U) \in C_0(\G)$ for each $i$.
For $a,b\in C_0(\G)$,
\begin{equation} 0 = \lim_{i\in \Ind} \big(U(a\otimes\xi_i) - a\otimes\xi_i
   \big| b\otimes\xi_i \big)
= \lim_i b^* (\id\otimes\omega_{\xi_i,\xi_i})(U) a - b^*a
= \lim_i b^* x_i a - b^*a.
\end{equation}
So for $\mu\in C_0(\G)^*$,
\begin{equation} \ip{\mu b^*}{a} = \lim_i \ip{\mu b^*}{x_i a}. \end{equation}
By Cohen Factorisation, every member of $C_0(\G)^*$ has the form $\mu b^*$,
so we conclude that $x_i a \rightarrow a$ weakly.  Similarly
$ax_i \rightarrow a$ weakly.  As the weak and norm closures of convex sets are equal, we can move to
a convex combination of the net $(x_i)_{i\in \mathcal{I}}$ and obtain a bounded approximate
identity for $C_0(\G)$.  Notice that convex combinations of the $x_i$ will
arise as slices of $\wW$ by convex combinations of states, that is,
by slicing against states.  Thus we obtain some new family of states
$(\lambda_j)_{j\in \mathcal{J}}$ such that $((\id\otimes\lambda_j)(\wW))_{j\in \mathcal{J}}$ is a bounded
approximate identity in $C_0(\G)$.

(iii)$\Longrightarrow$ (iv): Notice that equivalently $((\id\otimes\mu_i)(\wW)^*)_{i\in \mathcal{I}}
= ((\id\otimes\mu_i)(\wW^*))_{i\in \mathcal{I}}$ is an approximate identity for $C_0(\G)$.
Then, as in Section~\ref{sec:mults}, for each $i\in \Ind$ the element $a_i = (\id\otimes\mu_i)(\wW^*)$
represents a completely positive left multiplier of $L^1(\hh\G)$.

(iv)$\Longrightarrow$ (iii): This follows from \cite{daws}
as if $a_i\in C_0(\QG)$ represents a completely positive left multiplier of $L^1(\hh\G)$ then there is a state $\mu_i\in C_0^u(\hh \QG)^*$ such that $a_i = (\id\otimes\mu_i)(\wW^*)$.
\end{proof}

\begin{exam}
After a preprint version of this paper appeared, the equivalence (i)$\Longleftrightarrow$(iv) was used in \cite{Caspers} to prove the Haagerup property for the (non-amenable, non-discrete) locally compact quantum group arising as a quantum deformation of $SU(1,1)$.
\end{exam}

\begin{rem}
Property (ii) was proved for $\bz$ by P.Halmos in \cite{Halmos}. The equivalence of (i) and (ii) for classical locally compact groups is well known, dating back to \cite{BR}, where the harder direction (i)$\implies$(ii) is implicit in the proof of \cite[Theorem 2.5]{BR}. Property (ii) can also be compared with recent work of Brown and Guentner \cite{BG}, who characterise the Haagerup property for a discrete group in terms of the equality of the ``$C_0$-completion'' of the complex group algebra and the full group $C^*$-algebra, i.e.\ the fact there are enough $C_0$-representations to recover the universal norm. This was extended to the locally compact setting by Jolissaint  in \cite{JolC0}.
\end{rem}

\begin{prop} \label{HAP:subgroups}
Let $\H$ be a closed quantum subgroup of $\G$ in the sense of Woronowicz.
If $\G$ has the Haagerup property and is coamenable, then $\H$
has the Haagerup property.
\end{prop}
\begin{proof}
From condition (\ref{equivHAP:three}) in Theorem~\ref{equivHAP}, we can find
$(\mu_i)_{i\in \mathcal{I}}$ a net of states in $C_0^u(\hh\G)^*$ with
$(\id\otimes\mu_i)(\wW_{\G})$ a bounded approximate identity in $C_0(\G)$.  Let $\pi:C_0^u(\G)
\rightarrow C_0^u(\H)$ verify that $\H$ is a closed quantum subgroup of $\G$.
For each $i\in \Ind$ set $\lambda_i = \mu_i\circ\hh\pi$, where $\hh \pi: C_0^u(\hh\H) \to \M( C_0^u(\hh \G))$ is the dual morphism to $\pi$, so that $\lambda_i$ is a state on
$C_0^u(\hh\H)$.  Then
\[ (\id\otimes\lambda_i)(\wW_{\H})
= (\Lambda_{\H}\otimes\mu_i\hh\pi)(\WW_{\H})
= (\Lambda_\H\pi\otimes\mu_i)(\WW_\G). \]
Now, as $\G$ is coamenable, $C_0(\G) = C_0^u(\G)$, and so $\WW_\G
= \wW_\G$ and also $\pi$ can be considered as a map $C_0(\G)\rightarrow
C_0^u(\H)$.  Thus
\[ (\id\otimes\lambda_i)(\wW_{\H}) =
(\Lambda_\H\pi\otimes\mu_i)(\wW_\G)
= \Lambda_\H\pi( (\id\otimes\mu_i)(\wW_\G) ). \]
As $\pi$ is onto, it follows that $((\id\otimes\lambda_i)(\wW_{\H}))_{i \in \Ind}$ is an approximate identity for
$C_0(\H)$, verifying that $\H$ has the Haagerup property.
\end{proof}

\begin{rem}
Every classical $\G$ and every discrete $\G$ is coamenable, and this extra hypothesis is not excessive (in the way that asking for $\G$ to be amenable would be: compare Remark~\ref{rem:amen_not_enough}!)  To prove this result without assuming $\G$ coamenable seems tricky: we would, for example, wish to know that if $\mu\in C_0(\hh\G)^*$ is a state with $(\id\otimes\mu)(\wW_\G) \in C_0(\G)$ then also $(\id\otimes\mu)(\WW_\G) \in C_0^u(\G)$.  This is related to the following question: if $U$ is a mixing corepresentation of $C_0(\G)$, then following \cite[Proposition~6.6]{ku} we can ``lift'' $U$ to a unique corepresentation $V$ of $C_0^u(\G)$ with $(\Lambda_\G\otimes\id)(V)=U$.  Will $V$ still be mixing (that is, have $C_0$-coefficients)?
\end{rem}

Condition (iii) of Theorem \ref{equivHAP} corresponds to the existence  of a sequence/net of positive definite $C_0$-functions on a classical group $G$ converging to $1$ pointwise (see \cite{book}). Another equivalent formulation is the existence of a proper conditionally negative definite function $\psi$ on $G$ which, via Sch\"onberg's theorem, gives rise to a semigroup of positive definite $C_0$-functions $(e^{-t\psi(\cdot)})_{t>0}$ connecting the trivial representation (at $t=0$) to the regular representation (as $t\rightarrow\infty$). The analogous statement in the quantum setting is phrased in terms of convolution semigroups.

\begin{deft} \label{convssDef}
A \emph{convolution semigroup of states} on $C_0^u(\QG)$ is a family $(\mu_t)_{t\geq 0}$ of states on $C_0^u(\QG)$ such that
\begin{rlist}
\item $ \mu_{s+t} = \mu_s \star \mu_t := (\mu_s \ot \mu_t)\Com$  for all $s,t \geq 0$;
\item $\mu_0 = \Cou_u$;
\item $\mu_t(a) \stackrel{t \to 0^+}{\longrightarrow} \Cou_u(a)$ for each $a\in C_0^u(\QG)$.
\end{rlist}
\end{deft}

We thus have the following trivial consequence of Theorem \ref{equivHAP}.

\begin{prop} \label{convsemimplyHAP}
Let $\QG$ be a locally compact quantum group. If
there exists a convolution semigroup of states $(\mu_t)_{t\geq 0}$ on $C_0^u(\widehat{\QG})$ such that each $a_t:=(\id \ot \mu_t)(\wW)$  is an element of $C_0(\QG)$ and $a_t$ tends strictly to $1$ as $t\to 0^+$, then $\QG$ has the Haagerup property.
\end{prop}

Convolution semigroups of states have \emph{generating functionals} and are determined by them (see \cite{LS}). Thus to prove the converse implication it suffices to construct for a given quantum group with the Haagerup property a generating functional with certain additional properties guaranteeing that the resulting convolution semigroup satisfies the conditions above. However, in the general locally compact case it is difficult to decide whether a given densely defined functional is the generator of a convolution semigroup of bounded functionals. The situation is simpler if $\QG$ is discrete, which we return to in the next section.  A key task for us there will be to
see how much choice we have over the states which appear in Theorem~\ref{equivHAP}~(\ref{equivHAP:three}).  A first step in that programme is the following proposition.

\begin{prop}\label{prop:Rinv}
Let $\G$ have the Haagerup property.  Then there exists a net of states
$(\mu_i)_{i\in \mathcal{I}}$ on $C_0^u(\widehat{\QG})$ such that $\mu_i\circ\hh R_u = \mu_i$ for each $i\in \Ind$, and such that the net $\big((\id \ot \mu_i)(\wW)\big)_{i\in \Ind}$  is an approximate identity in $C_0(\QG)$.
\end{prop}
\begin{proof}
Pick a net of states $(\mu_i)_{i\in \mathcal{I}}$ as in Theorem~\ref{equivHAP} (\ref{equivHAP:three}), and let $\lambda_i = \frac12(\mu_i+\mu_i\circ \hh R_u)$
for each $i\in \Ind$.  As $\hh R_u$ is a $*$-anti-homomorphism, each $\lambda_i$ is state, and clearly $\lambda_i\circ\hh R_u=\lambda_i$.  By \cite[Proposition~7.2]{ku} we know that $(\hh R_u \otimes R)(\hh\Wuniv) = \hh\Wuniv$, and so $(R\otimes\hh R_u)(\wW) = \wW$, again using that $\hh R_u$ and $R$ are $*$-maps.  It follows that
\begin{equation} (\id \ot \mu_i\circ \hh R_u)(\wW)
= R\big( (\id\ot \mu_i)(\wW) \big) \in C_0(\G),
\end{equation}
as $R$ is an anti-automorphism of $C_0(\G)$.  Similarly, it follows easily that a net $(a_i)_{i\in \Ind}$ in $C_0(\G)$ is an approximate identity if and only if $(R(a_i))_{i\in \Ind}$ is, if and only if $(\frac12(a_i+R(a_i)))_{i\in \Ind}$ is.  Consequently, $\big((\id \ot \lambda_i)(\wW)\big)_{i\in \Ind}$ is indeed an approximate identity for $C_0(\G)$.
\end{proof}

\section{Haagerup approximation property for discrete quantum groups} \label{disc}

In this section $\QG$ will be a discrete quantum group, so that $\hh\QG$ is compact.  We present certain further equivalent characterisations of the Haagerup property in this case.

First observe the following consequence of Proposition \ref{coamenlcqg}.

\begin{prop}\label{amenHAP}
Every amenable discrete quantum group $\QG$ has the Haagerup property.
\end{prop}

\begin{proof}
This follows from Proposition \ref{coamenlcqg} and the fact that amenability of $\QG$ implies the coamenability of $\widehat{\QG}$, shown in \cite{Tomatsu}.
\end{proof}

Recall the notations of Subsection~\ref{compdiscsubsect}.  We will first provide a simple reinterpretation of condition (iii) appearing in Theorem~\ref{equivHAP} in the case of discrete $\QG$.  In Theorem~\ref{equivHAP}, for a state $\mu$ on $C_0^u(\hh\G)$, we considered the slice $(\id\ot\mu)(\wW)$.  As $\wW = \sigma(\hh\Wuniv^*)$, we can equivalently look at $(\mu\ot\id)(\hh\Wuniv)^*$, and clearly as far as the hypothesis of Theorem~\ref{equivHAP} are concerned, we may simply look at $\mc F\mu := (\mu\ot\id)(\hh\Wuniv)$.  When $\G$ is discrete, for $\alpha\in\Irr_{\widehat \QG}$ we shall write $(\mc F\mu)^\alpha \in \mathbb M_{n_\alpha}$ for the $\alpha$-component.  As noted in (\ref{unitcompact}),
\begin{equation}
\hh\Wuniv=\sum_{\alpha\in\Irr_{\hh\QG}}u_{ij}^\alpha\otimes e_{ij}\in \M(C^u(\hh\QG)\otimes c_0(\G)),
\end{equation}
so that $(\mc F\mu)^\alpha$ is the matrix with $(i,j)$ entry $\mu(u^\alpha_{ij})$.  Furthermore, this now gives us a natural interpretation of $\mc F\mu$ for a (possibly unbounded) functional $\mu$ on $\Pol(\hh\G)$.

\begin{prop}\label{lemma:hap_dis_case}
A discrete quantum group $\G$ has the Haagerup property if and only if there is a net of states
$(\mu_i)_{i\in \mathcal{I}}$ on $\Pol(\hh\G)$  such that:
\begin{enumerate}
\item for each $i\in \mathcal{I}$, we have that $((\mc F\mu_i)^\alpha)_{\alpha \in  \Irr_{\widehat{\QG}}} \ \in \prod_{\alpha \in \Irr_{\widehat \QG}}
\mathbb M_{n_\alpha}$ is actually in $\bigoplus_{\alpha \in \Irr_{\widehat \QG}} \mathbb M_{n_\alpha}$;
\item for each $\alpha \in \Irr_{\widehat \QG}$, the net $((\mc F\mu_i)^\alpha)_{i\in \mathcal{I}}$ converges in norm
to $1_{\mathbb M_{n_\alpha}}$.
\end{enumerate}
If the conditions above hold, the indexing set $\mathcal{I}$ can be chosen to be equal to $\bn$.
\end{prop}
\begin{proof}
As observed in Subsection~\ref{compdiscsubsect}, there is a bijection between states on $\Pol(\hh\G)$ and states on $C^u(\hh\G)$.  So by Theorem~\ref{equivHAP} (\ref{equivHAP:one})$\Longleftrightarrow$(\ref{equivHAP:three}), $\G$ has the Haagerup property if and only if there is a
net of states $(\mu_i)_{i \in \Ind}$ on $\Pol(\hh\G)$ such that $((\id\otimes
\mu_i)(\wW))_{i \in \Ind}$ is an approximate identity in $c_0(\G)$.  Via the discussion above, this is equivalent to the two stated conditions.

The last statement follows easily from the fact that $\Irr_{\widehat{\QG}}$ is countable.
\end{proof}

\subsection{The Haagerup property for $\QG$ via the von Neumann algebraic Haagerup approximation property for $L^{\infty}(\widehat{\QG})$}

Let $\mc M$ be a von Neumann algebra with a normal state $\phi$, and let
$(\pi,\Hil,\xi_0)$ be the GNS construction (when $\phi$ is faithful, we shall
tend to drop $\pi$).  Let $\Phi:\mc M\rightarrow\mc M$ be a unital completely positive map which preserves $\phi$.  Then, there is an induced map $T\in\mc B(\Hil)$ with $T(\pi(x)\xi_0) = \pi(T(x))\xi_0$.

The following definition of the von Neumann algebraic Haagerup approximation property is usually considered only for a von Neumann algebra equipped with a faithful
normal trace, and in that case, does not depend on the actual choice of such a trace (see \cite{j}). Here we propose the least restrictive possible extension to the case of general faithful normal states (for an extension reaching also the case of faithful normal semifinite weights and its properties see the forthcoming work \cite{CaSkal}).

\begin{deft} \label{HAPalg}
A von Neumann algebra $\mc M$ equipped with a faithful normal state $\phi$ is said to have the Haagerup approximation property (for $\phi$) if there exists a family of unital completely positive $\phi$-preserving normal maps $(\Phi_i)_{i\in \mathcal{I}}$ on $\mc M$ which converge to $\id_{\mc M}$ in the point $\sigma$-weak topology such that each of the respective induced  maps $T_i$ on $L^2(\mc M, \phi)$ is compact.  As each $T_i$ is a contraction, a standard argument shows that the point $\sigma$-weak convergence is equivalent to $T_i\stackrel{i \in \mathcal{I}}{\longrightarrow}1$ strongly on $L^2(\mc M,\phi)$.
\end{deft}

We use condition (iii) of Theorem \ref{equivHAP} to connect the Haagerup property for $\QG$ and the Haagerup approximation property for  $L^{\infty}(\hh \QG)$. On one hand the states featuring in Theorem \ref{equivHAP} can be used  to construct certain approximating multipliers on $L^{\infty}(\hh \QG)$; on the other given approximating maps on $L^{\infty}(\hh \QG)$ we can attempt to ``average'' them into multipliers and thus obtain states with the required properties.

For the following, recall that we denote the Haar state on $\hh \QG$ by $\hh \varphi$.

\begin{tw}\label{groupalg}
Let $\QG$ be a discrete quantum group. If $\QG$ has the Haagerup property, then $L^{\infty}(\hh \QG)$ has the Haagerup approximation property for $\hh\varphi$ (in the sense of Definition~\ref{HAPalg}).
\end{tw}
\begin{proof}
Let $(\mu_i)_{i \in \Ind}$ be a net of states in $C_0^u(\hh\G)^*$ satisfying condition (iii) in Theorem \ref{equivHAP}.  For each $i\in \Ind$ use the representation $\wW$ and
the states $\mu_i$ to build completely positive left multipliers $L_i:L^\infty(\hh\QG)\rightarrow L^\infty(\hh\QG)$ defined by
\begin{equation}
L_i(x)=(\id\otimes\mu_i)(\wW(x\otimes 1)\wW^*),\quad (x\in L^\infty(\hh{\QG})).
\end{equation}
By Proposition  \ref{states-multipliers}, each $L_i$ is a normal, unital,
$\hh\varphi$-preserving completely positive map, and induces an operator
$T_i$ on $L^2(\hh\G)$.  Furthermore, $T_i = (\id\otimes\mu_i)(\wW)$
which is a member of $c_0(\G)$ by condition (iii) in Theorem \ref{equivHAP}.  As $\G$ is discrete,
$c_0(\G)$ is the $c_0$-direct sum of matrix algebras, and so in particular
$T_i$ is a compact operator on $L^2(\hh\QG)$.  Finally, as $(T_i)_{i \in \Ind}$ is a bounded
approximate identity for $c_0(\G)$, $T_i\rightarrow 1_{L^2(\G)}$ strongly, as required.
\end{proof}

For the converse we need to start with normal, unital, $\hh\varphi$-preserving
completely positive maps $\Phi:L^\infty(\hh\G)\rightarrow L^\infty(\hh\G)$, and
``average'' these into multipliers. This can be achieved when $\G$ is discrete and unimodular. To this end we will use the following version of Theorem~5.5 of \cite{KrausRuan}, rephrased in our language for the convenience of the reader. Recall that the comultiplication of a locally compact quantum group is always injective.

\begin{tw}[Theorem 5.5, \cite{KrausRuan}] \label{KRtw}
Let $\QG$ be a discrete unimodular quantum group and let $\Phi:L^\infty(\hh\G)\rightarrow L^\infty(\hh\G)$ be a normal, completely positive, unital map. Let $E$ denote the unique $\hh \varphi$-preserving conditional expectation from $ L^\infty(\hh\G)\,\wot\,L^\infty(\hh\G)$ to $\hh \Delta (L^\infty(\hh\G))$ (whose existence follows from the fact that $\hh \varphi$ is a trace). Then the formula
\begin{equation}
 L = (\hh \Delta)^{-1} \circ E \circ (\Phi \ot \id_{L^\infty(\hh\G)}) \circ \hh\Delta
\label{KRaverage} \end{equation}
defines a normal unital completely positive map, which is (the adjoint of) a left multiplier on $L^\infty(\hh\G)$ represented by the element
\begin{equation} a = (\hh\varphi\otimes\id)( (\Phi\otimes\id)(\hh W) \hh W^* ) \in \M (c_0(\G)). \label{KRa}\end{equation}
\end{tw}

\begin{rem}
The multiplier $L$ above can be alternatively described by the following formula (which makes sense for an arbitrary locally compact quantum group $\QG$):
\begin{equation}\label{e6.4} L(x) = (\hh\varphi\otimes\id)
\big( \hh W ((\Phi\otimes\id)\hh\Delta(x)) \hh W^* \big),\qquad (x\in L^\infty(\hh\QG)).
\end{equation}
Using this alternative description, we can show that the suitably reformulated Theorem \ref{KRtw} remains valid if we only assume that $\G$ is unimodular (i.e.\ drop the discreteness assumption).
\end{rem}

We can now generalise Choda's result that the Haagerup property is a von Neumann property of a discrete group from \cite{choda} to the quantum setting.

\begin{tw} \label{groupalgKac}
Let $\QG$ be a discrete quantum group and assume that $\hh \QG$ is of Kac type. Then $\QG$ has the Haagerup property if and only if
 $L^{\infty}(\hh \QG)$ has the Haagerup approximation property.
\end{tw}

\begin{proof}
The ``only if'' claim is Theorem~\ref{groupalg}, so suppose that $L^\infty(\hh\G)$ has the Haagerup property, as witnessed by a net $(\Phi_i)_{i\in\mathcal I}$ of unital, completely positive, $\hh\varphi$-preserving normal maps, whose induced maps $(T_i)_{i\in\mathcal I}$ are compact operators on $L^2(\hh\G)$ with $\Phi_i\rightarrow\id_{L^\infty(\hh\QG)}$ point $\sigma$-weakly.

For each $i$, define the corresponding map $L_i$ by \eqref{KRaverage}. By Theorem \ref{KRtw}, this is the adjoint of a left multiplier $L_i$, which is ``represented'' by $a_i$, where $ a_i = (\hh\varphi\otimes\id)( (\Phi_i\otimes\id)(\hh W) \hh W^* ) \in \M (c_0(\G))$.  As $L_i$ is completely positive \cite[Proposition~4.2, Theorem~5.2]{daws} gives a unique positive functional $\mu_i\in C_0^u(\hh\G)^*$ with $a_i = (\id\otimes\mu_i)(\wW^*)$.

Write $\xi_0\in L^2(\hh\G)$ for the cyclic vector arising in the GNS-construction, so that the induced maps $T_i\in\mc B(L^2(\hh\G))$ satisfy $T_i(x\xi_0)=\Phi_i(x)\xi_0$ for $x\in L^\infty(\hh\G)$. This means that for $x,y\in L^\infty(\hh\G)$,
\begin{equation} \big( yT_ix\xi_0 \big|\xi_0 \big)
= \big( T_ix\xi_0 \big| y^*\xi_0 \big)
= \big( \Phi_i(x)\xi_0 \big| y^*\xi_0 \big)
= \hh\varphi\big( y \Phi_i(x) \big).
\end{equation}
Then, as $\hh W = \sigma(W^*)$,
\begin{equation}
a_i^* = (\id\otimes\hh\varphi)(W^*(\id\otimes\Phi_i)(W))
= (\id\otimes \omega_{\xi_0,\xi_0})
\big( W^*(1\otimes T_i)W \big).
\end{equation}
For $\eta_1,\eta_2\in L^2(\G)$ consider the rank-one operator $\theta_{\eta_1,\eta_2}=(\cdot|\eta_2)\eta_1$ on $L^2(\hh\G)=L^2(\G)$.
Then
\begin{equation} (\id\otimes \omega_{\xi_0,\xi_0})
\big( W^*(1\otimes \theta_{\eta_1,\eta_2})W \big)
= (\id\otimes\omega_{\eta_1,\xi_0})(W^*)
(\id\otimes\omega_{\xi_0,\eta_2})(W) \in c_0(\G). \end{equation}
Approximating the compact operators $T_i\in\mathcal B(L^2(\hh\G))$ by finite rank operators, we conclude that $a_i^*\in c_0(\G)$ for each $i$ (note that this could be alternatively obtained from the last statement of Theorem 5.5 (1) in \cite{KrausRuan}).

As $\G$ is discrete, $c_0(\G)^* = L^1(\G)$, and so for
$x\in c_0(\G),\ \omega\in c_0(\G)^*$,
\begin{equation}
\lim_{i\in\Ind}\ip{\omega}{a_ix}=\lim_{i\in\Ind}\ip{W^*(\id\otimes\Phi_i)(W)}{x\omega\otimes\hh\varphi}
= \ip{1\otimes 1}{x\omega \otimes \hh\varphi} = \ip{\omega}{x},
\end{equation}
since $\Phi_i\rightarrow \id_{L^\infty(\hh\QG)}$ point $\sigma$-weakly. So $a_i^*x\stackrel{i \in \Ind}{\rightarrow} x$ weakly, and similarly $xa_i^* \stackrel{i \in \Ind}{\rightarrow} x$ weakly for all $x\in c_0(\G)$. By Hahn-Banach, we can pass to a convex combination of the $(a_i)$ to obtain a bounded approximate identity for $c_0(\G)$ arising from slices of $\wW$ by states, just as at the end of the proof of Theorem \ref{equivHAP} (\ref{equivHAP:one})$\implies$(\ref{equivHAP:three}).  Theorem \ref{equivHAP} (iv)$\Longrightarrow$(i) then finishes the proof.\end{proof}

\begin{rem}
A $C^*$-algebraic version of the Haagerup approximation property for a pair $(\alg,\tau)$, where $\alg$ is a unital $C^*$-algebra and $\tau$ is a faithful tracial state on $\alg$, was introduced in \cite{Dong}. Using on one hand the fact that the $C^*$-algebraic Haagerup approximation property for $(\alg,\tau)$ passes to the analogous von Neumann algebraic property for  $\pi_{\tau}(\alg)''$ (see Lemma 4.5 in \cite{Suzuki}) and on the other the fact that the multipliers we construct in the proof of Theorem~\ref{groupalg} leave the  $C^*$-algebra $C(\widehat{\QG})$ invariant (see Remark~\ref{rem:mult_leaves_co_inv}) one can deduce  easily from the above theorem that a discrete unimodular quantum group $\QG$ has the Haagerup property if and only if the pair $(C(\widehat{\QG}), \hh\varphi)$ has Dong's $C^*$-algebraic Haagerup approximation property.

As remarked in Definition~\ref{HAPalg}, for the von Neumann algebraic Haagerup property the two natural notions of convergence, point-$\sigma$-weak on the algebra, or point-norm on the Hilbert space, agree.  For the $C^*$-algebraic Haagerup property, both \cite{Dong} and \cite{Suzuki} choose to work with point-norm convergence at the Hilbert space level, but one might also consider point-norm convergence at the algebra level; in general this will be a strictly stronger condition.  However, in our setting we do indeed get point-norm convergence in $C(\hh\G)$.  Indeed, in Theorem~\ref{groupalg} we construct multipliers $L_i$, each represented by $a_i\in c_0(\G)$, with $(a_i)$ a (contractive) approximate identity.  As argued in Remark~\ref{rem:mult_leaves_co_inv}, we see that for $\omega\in L^1(\G)$,
\begin{equation} L_i\big( (\id\otimes\omega)(\hh W) \big) = (\id\otimes\omega a_i)(\hh W)
\stackrel{i\rightarrow\infty}{\longrightarrow}
(\id\otimes\omega)(\hh W),
\end{equation}
the convergence being in norm.  As $C(\hh\G)$ is the norm closure of elements of
the form $(\id\otimes\omega)(\hh W)$, and each $L_i$ is a contraction, we conclude that $L_i(x) \rightarrow x$ for all $x\in C(\G)$, as claimed. We thank Caleb Eckhardt for an enquiry which prompted this observation.
\end{rem}

We can now record a number of example of quantum groups with the Haagerup property.

\begin{exam}[\cite{Brannan}, \cite{Brannan2}, \cite{Lemeux}]
The duals of the free orthogonal quantum groups, of the free unitary quantum groups, of the quantum automorphism groups of certain finite-dimensional $C^*$-algebras equipped with canonical traces, and of the quantum reflection groups $H_n^{s+}$ (the free wreath products $\bz_s\wr S_n^+$, see \cite{Julien}) for $n\geq 4$ and $1\leq s<\infty$ have the Haagerup property.
\end{exam}

The second corollary is related to cocycle twisted products of discrete quantum groups (studied for example in \cite{PierreLeonid}) with the Haagerup property.

\begin{cor} Let $\QG$ be a discrete unimodular quantum group and let $\Gamma$ be  a discrete abelian group
such that $C^*(\Gamma) \subset C^u(\hh\QG)$, with the inclusion intertwining the comultiplications. Let
$\sigma:\widehat{\Gamma} \times \widehat{\Gamma} \to \bt$ be a bicharacter. Then $\QG$ has the Haagerup property if and only if the twisted quantum group $\QG_{\sigma}$ has the Haagerup property.
\end{cor}
\begin{proof}
As the Haar state of $\hh{\QG}$ is a trace, it follows from \cite{PierreLeonid}  that the twisting does not modify the von Neumann algebra: $L^{\infty}(\hat{\QG}) = L^{\infty}(\hat{\QG}_{\sigma})$.
The corollary now follows from Theorem~\ref{groupalgKac}.
\end{proof}

We finish the subsection by exhibiting another corollary, related to the wreath product of compact quantum groups (\cite{Julien}) and also  to the considerations which will follow in Section~7.

\begin{cor} Let $\QG$ be a discrete unimodular quantum group. Let $\QH$ denote the dual of the free wreath product
product $\widehat{\QG} \wr S_2^+$. Then $\QG$ has the Haagerup property if and only if $\QH$ has the Haagerup property.
\end{cor}

\begin{proof}
Consider the algebra $C^u(\widehat{\QH})$. It is generated by commuting  copies of the $C^*$-algebras $C^u(\widehat{\QG}) \star C^u(\widehat{\QG})$ and $C(S^+_2) = C(\bz_2)$.
Thus  $C^u(\widehat{\QH}) \approx (C^u(\widehat{\QG}) \star C^u(\widehat{\QG})) \ot C(\bz_2)$.
Moreover it follows from \cite{Julien} that the Haar measure of $\widehat{\QH}$ is given with respect to this decomposition by the formula
$h = h_1 \ot h_2$, where $h_1$ is the free product of Haar states of $\widehat{\QG}$ and $h_2$ is induced by the Haar
measure of $\bz_2$. It follows that
$L^{\infty}(\widehat{\QH}) = (L^{\infty}(\widehat{\QG}) \star L^{\infty}(\widehat{\QG})  )\ot \bc^2$
which has the von Neumann algebraic Haagerup property if and only if $L^{\infty}(\widehat{\QG})$ has the von Neumann algebraic Haagerup property by \cite[Theorem~2.3]{j}. Theorem~\ref{groupalgKac} ends the proof.
\end{proof}

\begin{rem}The free wreath product of a coamenable compact quantum group by $S^+_2$ is not coamenable in general.
For example, taking $G = \ZZ$, the free wreath product of $\widehat{G}=\bt$ by $S^+_2$ is a non-coamenable compact quantum group (whose fusion rules are even non-commutative, see \cite{Julien}), and yet the above corollary shows that its dual has the Haagerup property.
\end{rem}

Some more examples of permanence of the Haagerup property with respect to constructions involving discrete unimodular quantum groups are given in Propositions \ref{amalgam} and \ref{HNN}.

\subsection{The Haagerup property via convolution semigroups of states and conditionally negative definite functions}
\label{sec:semigp_conv_states}

The quantum counterpart of a conditionally negative definite function is a generating function as defined below.

\begin{deft}
A \emph{generating functional} on $\hh \QG$ is a functional $L:\Pol (\hh \QG)
\rightarrow \bc$ which is selfadjoint, vanishes at $1_{\hh \QG}$ and is conditionally negative definite, i.e.\ negative on the kernel of the counit (formally: if $a\in \Pol(\widehat{\QG})$ and $\hh\Cou(a)=0$, then $L(a^*a) \leq 0$).
\end{deft}

The following fact can be viewed as a quantum version of Sch\"onberg's correspondence and goes back to the work of Sch\"urmann (see \cite{Schurmann}). In this precise formulation it can be deduced for example from \cite[Section~8]{QSCC2}.  Indeed, the correspondence between semigroups of hermitian functions, and self-adjoint $L$ with $L(1)=0$ follows easily from one-parameter semigroup theory, and bialgebra theory, see for example the sketch in \cite[Section~2]{LS}.  The harder part is to show that the extra property of being conditionally negative definite is enough to ensure a semigroup of states.

\begin{lem} \label{convgen}
There exists a one-to-one correspondence between:
\begin{rlist}
\item convolution semigroups of states $(\mu_t)_{t\geq 0}$ on $C^u(\widehat{\QG})$;
\item generating functionals $L$ on $\widehat{\QG}$;
\end{rlist}
given by
\begin{equation}\label{e6.24}
L(a) = \lim_{t\to 0^+} \frac{\hh\Cou(a) - \mu_t(a)}{t},\qquad (a\in \Pol(\widehat{\QG}) \subset C^u(\widehat{\QG}))\end{equation}
and
\begin{equation}\label{e6.25} \mu_t (a) = \exp_{\star}(-tL) (a):= \sum_{n=0}^{\infty} \frac{(-t)^n}{n!} L^{* n} (a),\qquad (a\in \Pol(\widehat{\QG}) \subset C^u(\widehat{\QG})).\end{equation}
\end{lem}

It will again be useful to consider the natural basis $(u^\alpha_{ij})$ of $\Pol(\G)$.  Given a generating functional $L$, let $(L^{\alpha})_{i,j} = L(u^{\alpha}_{ij})$ so that $L^\alpha$ is a $n_\alpha\times n_\alpha$ matrix. For $s>0$, write $a_s= (\mu_s \ot \id)(\hh\Wuniv) \in \M(c_0(\QG))$, where $(\mu_t)_{t\geq 0}$ is the convolution semigroup associated to $L$. As the map $C^u(\hh\G)^*\rightarrow \M(c_0(\G)); \mu\mapsto \mc F\mu = (\mu\ot\id)(\hh\Wuniv)$ is a homomorphism, it is easy to see that the $\alpha$-component of $a_s$ satisfies
\begin{equation}\label{e6.26} (a_s)^{\alpha} = \tu{e}^{-sL^{\alpha}},\qquad (\alpha \in \Irr_{\widehat{\QG}}).
\end{equation}
As a consequence we immediately obtain the following lemma, which is crucial in the sequel.

\begin{lem}\label{lemmastronggen}
Let $L$ and $(\mu_t)_{t\geq 0}$ be as in Lemma \ref{convgen}. Fix $s>0$. The following conditions are equivalent:
\begin{rlist}
\item the operator $a_s= (\mu_s \ot \id)(\hh\Wuniv) \in \M(c_0(\QG))$ belongs to $c_0(\QG)$;
\item the family of matrices $(\tu{e}^{-sL^{\alpha}})_{\alpha \in \Irr_{\hh\G}}$ (with $L^{\alpha}$ defined as above) converges to $0$ as $\alpha$ tends to infinity.
\end{rlist}
\end{lem}

\begin{deft} \label{stronglypropergen}
We call a generating functional $L$ \emph{symmetric} if the associated family of matrices $(L^{\alpha})_{\alpha \in \Irr_{\widehat{\QG}}}$ consists of self-adjoint matrices. A symmetric generating function $L$ is said to be \emph{proper} if it satisfies the following condition: for each $M>0$ there exists a finite set $F\subset \Irr_{\widehat{\QG}}$ such that for all $\alpha \in \Irr_{\widehat{\QG}} \setminus F$ we have that
\begin{equation} L^\alpha \geq M I_{n_{\alpha}}.
\end{equation}
\end{deft}

Note that the self-adjointness of the matrices $L^{\alpha}$ is equivalent to $L$ being $\hh S$-invariant, because $L(u^\alpha_{ij}) = L^\alpha_{ij}$ while $(L\circ \hh S)(u^\alpha_{ij}) = L((u^\alpha_{ji})^*) = \overline{ L^\alpha_{ji} }$.  This explains the use of the word `symmetric' in the definition. Moreover the assumption that a generating functional $L$ is $\hh S$-invariant implies immediately that each $L^{\alpha}$ is in fact a positive matrix.  This follows, as each $L^\alpha$ will be self-adjoint, and as each $\mu_s$ is a state, the norms of the matrices $(a_s)^{\alpha} = \tu{e}^{-sL^{\alpha}}$ do not exceed $1$; then observe that if $X$ is any self-adjoint matrix such that $(\tu{e}^{-sX})_{s>0}$ is a semigroup of contractions, then $X$ must be positive.  Note too that if $L$ is symmetric, then each of the states $\mu_t$ defined by (\ref{e6.25}) is $\hh S_u$-invariant.

In Proposition~\ref{prop:Rinv} we showed that the net of states $(\mu_i)_{i\in \Ind}$ appearing in Theorem~\ref{equivHAP} (\ref{equivHAP:three}) can be chosen to be $\hh R_u$-invariant.  We next show that, at least when $\QG$ is discrete, we can choose the states to be $\hh S_u$-invariant.

\begin{prop}\label{prop:Sinv}
Let $\QG$ be a discrete quantum group. Given an $\hh R_u$-invariant state $\mu\in C^u(\hh\QG)^*$, there exists an $\hh S_u$-invariant state $\nu\in C^u(\hh\QG)^*$ with
\begin{equation}\label{e6.28}
\|(\mathcal F\nu)^\alpha\|\leq\|(\mathcal F\mu)^\alpha\|,\text{ and }\|(\mathcal F\nu)^\alpha-I_{n_\alpha}\|\leq\|(\mathcal F\mu)^\alpha-I_{n_\alpha}\|,\qquad (\alpha\in\Irr_{\hh\G}).
\end{equation}
In particular if $\QG$ has the Haagerup property, then the net of states $(\mu_i)_{i\in\Ind}$ on $C^u(\widehat{\QG})$ such that  $\big((\mu_i\ot\id)(\hh\Wuniv)\big)_{i \in \Ind}$  forms an approximate identity in $c_0(\QG)$ obtained in Proposition \ref{prop:Rinv} can additionally be taken $\hh S_u$-invariant.
\end{prop}
\begin{proof}
Let $\mu\in C^u(\hh\QG)^*$ be an $\hh R_u$-invariant state.  Let $M\in L^\infty(\mathbb R)^*$ be an invariant mean--- so $M$ is a state, and if $f,g\in L^\infty(\mathbb R)$ and $t\in\mathbb R$ are such that $f(s)=g(s+t)$ for all $s\in \br$, then $\ip{M}{f} = \ip{M}{g}$.  Define $\nu\in C^u(\hh\QG)^*$ by
\begin{equation} \ip{\nu}{a} = \ip{M}{\big( \ip{\mu}{\hh\tau^u_t(a)} \big)_{t\in\mathbb R}} \qquad (a\in C^u(\hh\G)), \end{equation}
where $\{\hh\tau^u_t:t\in \br\}$ is the scaling automorphism group on $C^u(\hh\G)$ (see Section 9 of \cite{ku}).  As each $\hh\tau^u_t$ is a $*$-automorphism it follows that $\nu$ is a state, and by the invariance of $M$, it follows that $\nu\circ \hh\tau^u_s = \nu$ for all $s\in\mathbb R$.  We now use some elementary one-parameter group theory--- see \cite[Section~4.3]{kus} or \cite{kus3} for example.  As $\nu$ is invariant for $\{\hh\tau^u_t:t\in \br\}$ it follows that $\nu$ is analytic, and invariant for its extension to complex parameters, so in particular $\nu\circ \hh\tau^u_{-i/2} = \nu$.  As $\mu$ is $\hh R_u$-invariant, and each $\hh\tau^u_t$ commutes with $\hh R_u$, it follows that $\nu$ is $\hh R_u$-invariant.  Thus $\nu\circ \hh S_u = \nu \circ \hh R_u \circ \hh\tau^u_{-i/2} = \nu$.

By \cite[Proposition~9.1]{ku} we have that $(\hh\tau^u_t\otimes\tau_t)(\hh\Wuniv) = (\hh\Wuniv)$ for all $t\in \br$.  Let $\omega\in \ell^1(\G)$, and set
$a=(\id\ot\omega)(\hh\Wuniv)$, so
\begin{equation} \hh\tau^u_t(a) = (\id\ot\omega)\big( (\hh\tau^u_t\ot\id)(\hh\Wuniv) \big) = (\id\ot\omega)\big( (\id\ot\tau_{-t})(\hh\Wuniv) \big) = (\id\ot\omega\circ\tau_{-t})(\hh\Wuniv).
\end{equation}
The scaling group $\{\tau_t:t \in \br\}$ restricts to each matrix summand $\mathbb M_{n_\alpha}$ of $c_0(\G)$, a fact which is summarised in \cite[Section~2.2]{Tomatsu2}, for example.  Let $\{\tau^\alpha_t:t\in \br\}$ be the resulting group of automorphisms acting on $\mathbb M_{n_\alpha}$.  Let $p_\alpha:c_0(\G) \rightarrow \mathbb M_{n_\alpha}$ be the projection, so that $p_\alpha^*:\mathbb M_{n_\alpha}^* \rightarrow \ell^1(\G)$ is the inclusion.  Let $\omega = p_\alpha^*(\phi)$ for some $\phi\in \mathbb M_{n_\alpha}^*$. Then for all $\alpha \in \Irr_{\hh \G}$,
\begin{align}
\ip{\phi}{(\mc F\nu_i)^\alpha} &= \ip{\mc F\nu}{\omega}
= \ip{\nu}{(\id\ot\omega)(\hh\Wuniv)}
= \ip{M}{\big( \ip{\mu}{\hh\tau^u_t(a)} \big)_{t\in\mathbb R}} \nonumber\\
&= \ip{M}{ \big( \ip{\mu}{(\id\ot\omega\circ\tau_{-t})(\hh\Wuniv)} \big)_{t\in\mathbb R} }
= \ip{M}{\big( \ip{\mu}{(\id\ot p_\alpha^*(\phi\circ\tau^\alpha_{-t}))(\hh\Wuniv)} \big)_{t\in\mathbb R}} \nonumber\\
&= \ip{M}{\big( \ip{\mu \otimes \phi}{(\id\ot\tau^\alpha_{-t})(u^\alpha)} \big)_{t\in\mathbb R}}
= \ip{M}{\big( \ip{\phi}{\tau^\alpha_{-t}((\mc F\mu)^\alpha)} \big)_{t\in\mathbb R}}
\end{align}
where we consider $u^\alpha \in C^u(\G) \ot \mathbb M_{n_\alpha}$.  As $\phi$ was arbitrary, it follows that $\|(\mc F\nu)^\alpha\| \leq \|(\mc F\mu)^\alpha\|$.

As $\tau^\alpha_{-t}(I_{n_\alpha}) = I_{n_\alpha}$ for all $t\in \br$, it follows that
\begin{equation}\ip{\phi}{(\mc F\nu)^\alpha - I_{n_\alpha}}
= \ip{M}{\big( \ip{\phi}{\tau^\alpha_{-t}((\mc F\mu)^\alpha - I_{n_\alpha})} \big)_{t\in\mathbb R}}
\end{equation}
Hence $\| (\mc F\nu)^\alpha - I_{n_\alpha} \|\leq\|(\mathcal F\mu)^\alpha-I_{n_\alpha}\|$, establishing (\ref{e6.28}).

When $\G$ has the Haagerup property let $(\mu_i)_{i\in \Ind}$ be a net of $\hh R_u$-invariant states, as given by Proposition~\ref{prop:Rinv} combined with Proposition~\ref{lemma:hap_dis_case}.  Thus, for each fixed $i\in \Ind$, we have that $\|(\mc F\mu_i)^\alpha\|\rightarrow 0$ as $\alpha\rightarrow\infty$, and for each fixed $\alpha\in \Irr_{\hh \G}$, we have that $\| (\mc F\mu_i)^\alpha - I_{n_\alpha}\|\stackrel{i\in \Ind}{\rightarrow} 0$.  For each $i$, form $\nu_i$ from $\mu_i$ as above. Then the net $(\nu_i)$ will satisfy the same limiting conditions.
\end{proof}

\begin{tw}\label{generators}
Let $\QG$ be a discrete quantum group.  The following are equivalent:
\begin{rlist}
\item $\QG$ has the Haagerup property;
\item there exists a convolution semigroup of states $(\mu_t)_{t\geq 0}$ on $C_0^u(\widehat{\QG})$ such that each $a_t:=(\mu_t \ot \id)(\hh\Wuniv)$  is an element of $c_0(\QG)$, and $a_t$ tend strictly to $1$ as $t\to 0^+$;
\item $\widehat{\QG}$ admits a symmetric proper generating functional.
\end{rlist}
When these conditions hold, the semigroup of states in (ii) can additionally be chosen to be $\hh S_u$-invariant.
\end{tw}
\begin{proof}
(iii)$\Longrightarrow$ (ii):
If $L$ is a symmetric proper generating functional on $\hh\G$, then condition (ii) in Lemma \ref{lemmastronggen} is satisfied so that the operators $a_t$ arising from the corresponding semigroup of states $(\mu_t)_{t\geq 0}$ given by Lemma \ref{convgen} lie in $c_0(\G)$. As $\QG$ is discrete, the strict convergence in $c_0(\QG)$ is the same as  convergence of the individual entries of the corresponding matrices. For each fixed $\alpha$, (\ref{e6.26}) gives
 $(a_t)^\alpha = \tu{e}^{-tL^\alpha}$, so that $(a_t)^\alpha \rightarrow I_{n_\alpha}$ as
$t\rightarrow 0^+$ as required.  Further, as $L$ is symmetric, the states $(\mu_t)_{t\geq 0}$ are $\hh S_u$-invariant.

(ii)$\Longrightarrow$ (i):
Follows from Proposition~\ref{convsemimplyHAP}.

(i)$\Longrightarrow$ (iii):
Choose $(F_n)_{n=1}^{\infty}$, an increasing sequence of finite subsets of $\Irr_{\widehat{\QG}}$ such that $\bigcup_{n=1}^{\infty} F_n = \Irr_{\widehat{\QG}}$, a sequence $(\epsilon_n)_{n=1}^{\infty}$ of positive numbers tending to $0$ and a sequence $(\beta_n)_{n=1}^{\infty}$ of positive numbers increasing to infinity such that $\sum_{n=1}^{\infty} \beta_n \epsilon_n <\infty$.  For each $n\in\bn$ use Proposition \ref{prop:Sinv} to find an $\hh S_u$-invariant state $\mu_n$ such that $\mc F\mu_n \in c_0(\G)$ satisfies  \begin{equation}\|I_{n_{\alpha}} -(\mathcal{F} \mu_n)^{\alpha}\| \leq \epsilon_n\qquad (\alpha\in F_n). \label{eq:mattfour}
\end{equation}
Each matrix $(\mathcal{F} \mu_n)^{\alpha}$ is contractive, and as $\mu_n$ is $\hh S_u$-invariant, it is also self-adjoint.

Define $L:\Pol(\widehat{\QG})\to \bc$ by
\begin{equation}\label{e6.34}
 L = \sum_{n=1}^{\infty} \beta_n (\Cou - \mu_n),
\end{equation}
with the convergence understood pointwise. We claim that $L$ is a (well-defined) symmetric proper generating functional.
Note first that for any $\alpha \in \Irr_{\widehat{\QG}}$, we have that
\begin{equation}\label{e6.35} \big( L(u_{ij}^{\alpha}) \big)_{i,j} = \sum_{n=1}^{\infty} \beta_n (I_{n_{\alpha}} - (\mathcal{F}\mu_n)^{\alpha})
\end{equation}
and the convergence of this sum is guaranteed by the fact that there exists $N\in \bn$ such that $\alpha \in F_n$ for all $n\geq N$, and using
(\ref{eq:mattfour}).  Hence the sum in (\ref{e6.34}) does converges pointwise.  As $L$ is a sum of self-adjoint functionals and further, on the kernel of the counit it is a sum of states multiplied by non-positive scalar coefficients, it is a generating functional. Hence it remains to show that $L$ is symmetric and proper. Observe that for each $\alpha \in \Irr_{\widehat{\QG}}$ and $n \in \bn$ we have that
$L^{\alpha} \geq \beta_n (I_{n_{\alpha}} - (\mc F\mu_n)^{\alpha})$. Thus it suffices for a given $M>0$ to choose $n\in \bn$ so that $\beta_n >2M$ and note that as  $\mathcal{F} \mu_n \in c_0(\QG)$, there exists a finite set $F\subset \Irr_{\widehat{\QG}}$ such that for $\alpha \in \Irr_{\widehat{\QG}} \setminus F$ we have that $\|(\mathcal{F}\mu_n)^{\alpha}\| \leq \frac{1}{2}$, so also $ I_{n_{\alpha}} - (\mathcal{F}\mu_n)^{\alpha} \geq \frac{1}{2} I_{n_{\alpha}}$ (recall that the matrix $(\mathcal{F}\mu_n)^{\alpha}$ is self-adjoint). Hence for $\alpha \in \Irr_{\widehat{\QG}} \setminus F$,
\begin{equation} L^{\alpha} \geq 2M \frac{1}{2} I_{n_{\alpha}},
\end{equation}
as required.
\end{proof}

The following concept generalises cocycles for unitary representations of classical discrete groups.

\begin{deft}
Let $\QG$ be a discrete quantum group, let $\Hil$ be a Hilbert space and let $\pi:C_u(\widehat{\QG})\to \mc B(\Hil)$ be a (unital) representation. A map $c:\Pol(\widehat{\QG})\to \Hil$ is called a \emph{cocycle} for $\pi$ if
\begin{equation} c(ab) = \pi(a) c(b) + c(a)\hat{\Cou}(b),\qquad (a,b \in \Pol(\widehat{\QG})).
\end{equation}
More generally a map $c:\Pol(\widehat{\QG})\to \Hil$ is said to be a cocycle on $\QG$ if there exists a representation $\pi:C_u(\widehat{\QG})\to \mc B(\Hil)$ such that $c$ is a cocycle for $\pi$.
\end{deft}

A cocycle determines a family of matrices $(c^{\alpha})_{\alpha \in \Irr_{\hh\G}}$ with entries in $\Hil$ given by
\begin{equation} c^{\alpha}=  (c(u^{\alpha}_{ij}))_{i,j=1}^{n_{\alpha}} \in M_{n_{\alpha}}(\Hil)\cong \mc B(\bc^{n_{\alpha}}, \bc^{n_{\alpha}} \ot \Hil)
\qquad (\alpha\in\Irr_{\hh\G}),
\end{equation}
so that we view each $c^{\alpha}$ as an operator between Hilbert spaces.

The following definitions are due to Vergnioux. Again, they extend classical notions for cocycles on discrete groups.

\begin{deft}
Let $\QG$ be a discrete quantum group and let $c:\Pol(\widehat{\QG})\to \Hil$ be a cocycle on $\QG$. Then $c$ is said to be \emph{bounded} if the family of operators  $(c^{\alpha})_{\alpha \in \Irr_{\widehat{\QG}}}$ is bounded; it is said to be \emph{proper} if for each $M>0$ there exists a finite set $F\subset \Irr_{\hh\G}$ such that
\begin{equation} (c^\alpha)^*c^{\alpha} \geq M I_{n_{\alpha}}
\qquad (\alpha \in \Irr_{\hh\G} \setminus F).
\end{equation}
Finally $c$ is said to be \emph{real} if
\begin{equation}\label{e6.40}
( c(a) | c(b) ) = ( c(\widehat{S_u}(b^*)) | c(\widehat{S_u}(a)^*) )
\qquad (a,b\in \Pol(\widehat{\QG})).
\end{equation}
As $\widehat{\Cou} = \widehat{\Cou} \circ \widehat{S_u}$, it is enough to verify (\ref{e6.40}) for $a,b \in \Pol (\widehat{\QG}) \cap \tu{Ker}(\widehat{\epsilon})$.
\end{deft}

The next result sets out the connection between (real) cocycles on $\QG$ and generating functionals on $\widehat{\QG}$.

\begin{prop}[\cite{QSCC1}, \cite{Kyed}]                 \label{L versus c}
Let $\QG$ be a discrete quantum group. If $L$ is a generating functional on $\widehat{\QG}$ then there exists a cocycle $c$ on $\QG$  such that
\begin{equation}  \label{Lcocycle}
( c(a) | c(b) ) = - L(b^*a) + \widehat{\Cou}(a) \overline{L(b)} + L(a) \overline{\widehat{\Cou}(b)} \qquad (a,b \in \Pol (\widehat{\QG})). \end{equation}
On the other hand if $c$ is a real cocycle on $\QG$, then one obtains a generating functional $L$ on $\hh\QG$ such that (\ref{Lcocycle}) holds, by defining
\begin{equation} L(a) = \frac{1}{2} \sum_{i=1}^n\big( c(a_{(1,i)}) \big| c(\hh S_u(a_{(2,i)}^*)) \big), \label{Verg_form_L} \end{equation}
where $a\in\Pol(\hh\G)$ has $\hh\Delta(a)=\sum_{i=1}^na_{(1,i)}\otimes a_{(2,i)}$.  \end{prop}
\begin{proof}
The construction of $c$ follows from a GNS type construction for $L|_{\tu{Ker}\, \widehat{\Cou}}$, see for example \cite[Section~6]{QSCC1} (note that a different linearity convention for scalar produces is used in \cite{QSCC1}). The construction of (\ref{Verg_form_L}) is Theorem 4.6 in \cite{Kyed}, attributed in that paper to Vergnioux.
\end{proof}

Note that if a generating functional $L$ on $\widehat{\QG}$ and a cocycle $c$ on $\G$ are related via the formula \eqref{Lcocycle}, then it is easily checked that $c$ is real if and only if $L$ is invariant under $\widehat{S}_u$ on $(\tu{Ker}\, \widehat{\epsilon}) ^2=\{b^*a:a,b\in\tu{Ker}\, \widehat{\epsilon}\}$. The following proposition shows that in fact if $c$ is real, and $L$ is constructed according to the formula \eqref{Verg_form_L}, then $L=L\circ \widehat{S}_u$.

\begin{prop}\label{real_cocycle_gives_S_inv}
Let $c$ be a real cocycle on $\QG$ and $L$ be the generating functional on $\Pol(\widehat{\G})$ given by~\eqref{Verg_form_L}. Then $L \circ \hh S_u= L$.
\end{prop}
\begin{proof}
Note that $c$ and $L$ satisfy \eqref{Lcocycle}.
Fix $a \in \Pol(\hh \G)$ and write $\hh\Delta(a)=\sum_{i=1}^na_{(1,i)}\otimes a_{(2,i)}$.  Then, by (\ref{Verg_form_L}) and (\ref{Lcocycle}),
\begin{align} 2L(a) & = \sum_{i=1}^n\big( c(a_{(1,i)}) \big| c(\hh S_u(a_{(2,i)}^*)) \big) \nonumber\\
&= \sum_{i=1}^n\Big(-L\big( \hh S_u(a_{(2,i)}^*)^* a_{(1,i)}\big) + \hh{\Cou}(a_{(1,i)})
\overline{ L(\hh S_u(a_{(2,i)}^*)) } + L(a_{(1,i)})
\overline{ \hh{\Cou}(\hh S_u(a_{(2,i)}^*)) } \Big).
\end{align}
Recalling that $\hh S_u(b^*)^*=\hh S_u^{-1}(b)$ for $b\in\Pol(\hh\G)$, and that $\hh\Cou=\hh\Cou\circ\hh S_u$, this gives
\begin{equation}\label{e6.44}
2L(a)=\sum_{i=1}^n\Big(-L(\hh S_u^{-1}(a_{(2,i)})a_{(1,i)})+\hh\Cou(a_{(1,i)})L(\hh S_u^{-1}(a_{(2,i)}))+L(a_{(1,i)})\hh\Cou(a_{(2,i)})\Big).
\end{equation}
As $\Pol(\G)$ is a Hopf $*$-algebra,
\begin{equation}\label{e6.45}
\hh{\Cou}(a)1 = \sum_{i=1}^n\hh S_u(a_{(1,i)}) a_{(2,i)}=
\sum_{i=1}^n\hh S_u^{-1}(a_{(2,i)}) a_{(1,i)}
\end{equation}
and
\begin{equation}\label{e6.46}
\sum_{i=1}^n\hh{\Cou}(a_{(1,i)}) a_{(2,i)} = a = \sum_{i=1}^na_{(1,i)} \hh{\Cou}(a_{(2,i)})
\end{equation}
Combining (\ref{e6.45}) and (\ref{e6.46}) with (\ref{e6.44}) gives
\begin{equation} 2L(a) = -L\big( \hh{\Cou}(a)1 \big) + L(\hh S_u^{-1}(a)) + L(a)
= L(\hh S_u^{-1}(a)) + L(a),
\end{equation}
as $L(1)=0$. Thus it follows that $L\circ \hh S_u = L$, as claimed.
\end{proof}

The next result characterises the Haagerup property for a discrete  quantum group $\QG$ via the existence of a proper real cocycle on $\QG$. As explained in the introduction, it should be thought of as the counterpart of the characterisation of the Haagerup property of a discrete group $\Gamma$ via the existence of the proper affine isometric action on a real Hilbert space. We thank Roland Vergnioux for the clarifying remarks which led us to this result.

\begin{tw}        \label{HAP-cocycle}
Let $\QG$ be a discrete quantum group. Then  $\QG$ has the Haagerup property if and only if it admits a proper real cocycle.
\end{tw}

\begin{proof}
Suppose first that we are given a generating functional $L$ on $\widehat{\QG}$ and a cocycle $c$ on $\G$ related via the formula \eqref{Lcocycle}. Fix $\alpha \in \Irr_{\widehat{\QG}}$ and compute the matrix $(c^{\alpha})^* c^{\alpha}$:
\begin{align}\big((c^{\alpha})^* c^{\alpha}\big)_{ij} & = \sum_{k=1}^{n_{\alpha}} (c(u^\alpha_{kj})| c(u^\alpha_{ki}))
=   \sum_{k=1}^{n_{\alpha}} \big( - L(u_{ki}^\alpha{}^* u_{kj}^\alpha) + \widehat{\Cou} (u_{kj}^\alpha) \overline{L(u_{ki}^\alpha)} + L (u_{kj}^\alpha) \overline{\widehat{\Cou}(u_{ki}^\alpha)} \big)
\nonumber\\&=
- L( \sum_{k=1}^{n_{\alpha}} u_{ki}^\alpha{}^* u_{kj}^\alpha) + \sum_{k=1}^{n_{\alpha}}  \big(\delta_{kj} \overline{L(u^\alpha_{ki})} + L (u^\alpha_{kj}) \delta_{ki} \big) \nonumber\\
&= - L(\delta_{ij} 1) + L(u^\alpha_{ji}{}^*) + L (u_{ij}^\alpha)
= (L^{\alpha} + (L^{\alpha})^*)_{ij}.
\end{align}

Suppose then that $\QG$ has the Haagerup property. Theorem~\ref{generators} implies that $\widehat{\QG}$ admits a proper symmetric generating functional $L$.
The first part of Proposition \ref{L versus c} implies that there exists a cocycle $c$ on $\QG$ related to $L$ via the formula \eqref{Lcocycle}. Since $L$ is symmetric, it is $\hh S_u$-invariant so certainly $\hh S_u$-invariant on $(\tu{Ker}\, \widehat{\epsilon}) ^2$. Hence, as noted after Proposition \ref{L versus c}, $c$ is real.  As $(c^{\alpha})^* c^{\alpha} = 2 L^{\alpha}$ for each $\alpha$, properness of $L$ is equivalent to properness of $c$.

Conversely, if $\QG$ admits a proper real cocycle $c$, then by the second part of Proposition \ref{L versus c} there exists a generating functional $L$ on $\widehat{\QG}$ related to  $c$ given by (\ref{Verg_form_L}). By Proposition~\ref{real_cocycle_gives_S_inv} this generating functional satisfies $L\circ\hh S_u=L$, and so by the remark after Definition~\ref{stronglypropergen}, each $L^\alpha$ is self-adjoint.  That is, $L$ is symmetric. Again, $(c^{\alpha})^* c^{\alpha} = 2 L^{\alpha}$ for each $\alpha$, and so properness of $c$ implies properness of $L$, and hence $\G$ has the Haagerup property by Theorem~\ref{generators}.
\end{proof}

We finish this subsection by proving two lemmas which will be needed in the last section of the paper.

\begin{lem} \label{lem:boundedperturbstates}
Let $\QG$ be a discrete  quantum group which has the Haagerup property. Then there exists a sequence of states
$(\mu_k)_{k\in \bn}$ on $\Pol(\hh\G)$  such that:
\begin{rlist}
\item for each $k \in \bn$ the family of matrices $((\mathcal F\mu_k)^\alpha)_{\alpha \in \Irr_{\widehat{ \QG}}}$
belongs to  $\bigoplus_{\alpha \in \Irr_{\widehat \QG}} \mathbb M_{n_\alpha}$;
\item for each $\alpha \in  \Irr_{\widehat \QG}$, the sequence $((\mathcal F\mu_k)^\alpha)_{k\in \bn}$ converges in norm
to the identity matrix in $\mathbb M_{n_\alpha}$;
\item for each $k \in \bn$ and $\alpha \in \Irr_{\widehat{\QG}}$ with $\alpha \neq 1$, we have that $\|\mu_k^\alpha\| \leq \exp(-\frac{1}{k})$.
\end{rlist}
\end{lem}

\begin{proof}
The proof is based on adjusting the output of Proposition~\ref{lemma:hap_dis_case} in order to obtain the extra third condition. Choose a sequence $(\omega_k)_{k\in \bn}$ of states   on $\Pol(\widehat{\QG})$ satisfying the conditions in that proposition. Let $h$ denote the Haar state of $\hh\QG$ and put  $L=\hh\Cou - h$. It is easy to see that $L$ is a (bounded) generating functional, and moreover for each $t\geq 0$ we have that $\exp_{\star}(-tL) = \textup{e}^{-t} \hh \Cou + (1 - \textup{e}^{-t})h$.   Put $\psi_k = \exp_{\star} (-\frac{1}{k}L)$. It is easy to check that the sequence $(\mu_k)_{k\in \bn}$, where $\mu_k = \omega_k \star \psi_k$, $k\in \bn$, satisfies the required conditions (note that the norm of each matrix $\omega_k^\alpha$ is not greater then $1$, as $\omega_k$ is a state, and that $(\mathcal{F}\mu_k)^{\alpha} = (\mathcal{F}\omega_k)^{\alpha} (\mathcal{F}\psi_k)^{\alpha}$).
\end{proof}

The above lemma has a natural counterpart for generating functionals on discrete quantum groups with the Haagerup property.

\begin{lem} \label{lem:boundedperturb}
Let $\QG$ be a discrete quantum group which has the Haagerup property.
 Then $\widehat{\QG}$ admits a symmetric proper generating functional $L:\Pol(\widehat{\QG})\to \bc$ such that
for each $\alpha \in \Irr_{\widehat{\QG}}$ we have that
\begin{equation}L^{\alpha} \geq I_{n_{\alpha}}.\end{equation}
\end{lem}
\begin{proof}
Let $L_1:\Pol(\widehat{\QG})\to \bc$ be a symmetric proper generating functional and let $L$ be defined  as in the proof of Lemma \ref{lem:boundedperturbstates}. As the sum of generating functionals is a generating functional, it suffices to consider $L + L_1$.
\end{proof}

\section{Free product of discrete quantum groups with the Haagerup property has the Haagerup property}
\label{sec:free_prods}

In this  section we apply the techniques developed earlier to generalise Jolissaint's result that the Haagerup property for discrete groups is preserved under taking free products (see \cite{Jol0} or \cite{book} for two different proofs, and note that in fact the Haagerup property is also preserved under taking a free product with amalgamation over a finite subgroup).  In the case of discrete unimodular quantum groups the shortest way to this theorem is via the von Neumann algebraic Haagerup approximation property and the fact it is preserved under taking free products of finite von Neumann algebras (see \cite{j}); this method can be also used to establish the quantum version of the result, mentioned above, for the free product with amalgamation over a finite (quantum) subgroup (see Proposition~\ref{amalgam}). The proof we present in the general case is closer in spirit to the classical proof in \cite{book}; the techniques developed may also be of interest in other contexts.

Recall first the definition of the free product of discrete quantum groups, originally introduced by S.\,Wang in \cite{wang}.  For compact quantum groups $\QG_1$, $\QG_2$, the $C^*$-algebra $C^u(\QG_1) * C^u(\QG_2)$ (the usual $C^*$-algebraic free product of unital $C^*$-algebras with amalgamation over the scalars) has a natural structure of the algebra of functions on a compact quantum group, with the coproduct arising from the universal properties of the free product applied to the maps $(\iota_1 \ot \iota_1) \Com_1$ and $(\iota_2 \ot \iota_2) \Com_2$, where $\Com_1, \Com_2$ denote the respective (universal) coproducts of $\QG_1$ and $\QG_2$ and $\iota_1:C^u(\QG_1) \to  C^u(\QG_1) * C^u(\QG_2)$, $\iota_2:C^u(\QG_2) \to  C^u(\QG_1) * C^u(\QG_2)$ are the canonical injections. We call the resulting compact quantum group \emph{the dual free product} of $\QG_1$ and $\QG_2$ and denote it by $\QG_1 \widehat{*} \, \QG_2$, so that $C^u(\QG_1) * C^u(\QG_2) = C^u(\QG_1 \widehat{*}\, \QG_2)$. The following result of Wang is crucial for working with the dual free products.

\begin{tw}[Theorem 3.10 of \cite{wang}] \label{repfree}
Let $\QG_1$, $\QG_2$ be compact quantum groups. Then $\ \Pol(\QG_1 \widehat{*}\, \QG_2)= \Pol(\QG_1) * \Pol(\QG_2)$ (where on the right hand side we have the $^*$-algebraic free product of unital algebras, identifying the units) and
\begin{equation} \Irr_{\QG_1 \widehat{*}\, \QG_2} =1 \cup \{U^{\alpha_1}\tp \cdots \tp U^{\alpha_k}: k\in \bn, i(j)\in \{1,2\}, i(j)\neq i(j+1), \alpha_{j} \in \Irr_{\QG_{i(j)}}, U^{\alpha_j} \neq 1 \},\end{equation}
where $U^{\alpha_1}\tp \cdots \tp U^{\alpha_k} \in M_{n_{\alpha_1}} \ot \cdots \ot  M_{n_{\alpha_k}} \ot (\Pol(\QG_1) * \Pol(\QG_2))$,
\begin{equation} (U^{\alpha_1}\tp \cdots \tp U^{\alpha_k})_{(l_1,\ldots,l_k),(m_1,\ldots,m_k)} = u^{\alpha_1}_{l_1,m_1} \cdots u^{\alpha_k}_{l_k,m_k} .\end{equation}
The Haar state of $\QG_1 \widehat{*}\, \QG_2$ is the free product of the Haar states of $\QG_1$ and $\QG_2$.
\end{tw}

Note that the last statement of the above theorem implies in particular that $L^{\infty}(\QG_1 \widehat{*}\, \QG_2) \cong L^{\infty}(\QG_1) *  L^{\infty}(\QG_2)$, where this time $*$ denotes the von Neumann algebraic free product (with respect to the Haar states of the respective $L^{\infty}$-algebras, see for example \cite{vdn}).

\begin{deft}
Let $\QG_1$, $\QG_2$ be discrete quantum groups. The free product of $\QG_1$ and $\QG_2$ is the discrete quantum group $\QG_1 * \QG_2$ defined by the equality
\begin{equation}\QG_1 * \QG_2 = \widehat{ \widehat{\QG}_1 \widehat{*}\, \widehat{\QG}_2}.
\end{equation}
\end{deft}

One may check  that the notion above is compatible with the notion of the free product of classical discrete groups (recall that if $\Gamma_1,\Gamma_2$ are discrete groups, then $C(\widehat{\Gamma_i})=C^*(\Gamma_i)$ and $C^*(\Gamma_1 * \Gamma_2) \cong C^*(\Gamma_1) * C^*(\Gamma_2)$).
It is also easy to observe that the free product of unimodular discrete quantum groups is unimodular (as the free product of tracial states is tracial). Finally we record the following well-known and easy observation.

\begin{prop} \label{subgroupfree}
Let $\QG_1$, $\QG_2$ be discrete quantum groups. Then both $\QG_1$ and $\QG_2$ are closed quantum subgroups of $\QG_1 * \QG_2$.
\end{prop}
\begin{proof}
Recall that  $L^{\infty}(\widehat{\QG_1} \widehat{*}\, \widehat{\QG_2}) \cong L^{\infty}(\widehat{\QG_1}) *  L^{\infty}(\widehat{\QG_2})$. It is easy to check that the canonical injection of $L^{\infty}(\widehat{\QG_1})$ into $L^{\infty}(\widehat{\QG_1}) *  L^{\infty}(\widehat{\QG_2})$ is a normal unital $^*$-homomorphism intertwining the respective coproducts. This means that $\QG_1$ is a closed quantum subgroup of $\QG_1 * \QG_2$.  The case of $\QG_2$ follows identically.
\end{proof}

\begin{rem}
Note that the terminology introduced here, used earlier for example in \cite{BanicaSkalski}, is different from that of \cite{wang}, where the author called the free product what we call the \emph{dual} free product of compact quantum groups. The advantage of the notation and nomenclature employed here is that it is consistent with the free product of classical discrete groups and also with the results such as the one stated above.
\end{rem}

Before we begin the proof of the main theorem of this section we need to introduce another construction: that of a \emph{c-free} (\emph{conditionally free}) product of states, introduced in \cite{BozSpe} and later studied for example in \cite{BLS}. Here we describe it only in the case of two algebras. Observe first that if $\Alg_1, \Alg_2$ are unital $^*$-algebras equipped respectively with states (normalised, hermitian, positive functionals) $\psi_1$ and $\psi_2$, then the $^*$-algebraic free product $\Alg_1 * \Alg_2$ can be identified (as a vector space) with the direct sum
\begin{equation} \bc 1 \oplus \bigoplus_{n=1}^{\infty} \bigoplus_{i(1)\neq \cdots \neq i(n)} \Alg_{i(1)}^{\circ} \ot \cdots \ot \Alg_{i(n)}^{\circ},
\end{equation}
where $i(j) \in \{1,2\}$ and $\Alg_1^{\circ} = \Ker \psi_1$, $\Alg_2^{\circ} = \Ker \psi_2$. This ensures that the following definition makes sense.

\begin{deft}
Let $\Alg_1, \Alg_2$ be unital $^*$-algebras equipped respectively with states $\psi_1$ and $\psi_2$.
Let $\phi_1$ and $\phi_2$ be two further states respectively on $\Alg_1$ and on $\Alg_2$. The conditional free product of $\phi_1$ and $\phi_2$ is the functional $\omega:=\phi_1 *_{(\psi_1,\psi_2)} \phi_2$ on $\Alg_1 * \Alg_2$ defined by the prescription $\omega(1) = 1$ and
 \begin{equation} \omega (a_1 \cdots a_n) = \phi_{i(1)} (a_1) \cdots \phi_{i(n)} (a_n)
 \end{equation}
 for all $n \in \bn$, $i(1)\neq \cdots \neq i(n)$ elements in $\{1,2\}$ and $a_j \in \Ker \psi_{i(j)}$ for $j=1,\ldots,n$.
\end{deft}

The crucial property of the conditional free product of states is that it is again a state (Theorem 2.2 of \cite{BLS}; the result follows also from Theorem 3.1 in \cite{BocaFree}).  Given two compact quantum groups $\QG_1$ and $\QG_2$ with Haar states $h_1$ and $h_2$, and two further states $\phi_1$, $\phi_2$ on, respectively, $\Pol(\QG_1)$ and $\Pol(\QG_2)$, write $\phi_1 \diamond \phi_2$ for a respective conditionally free product:
\begin{equation} \phi_1 \diamond \phi_2 : = \phi_1 *_{(h_1,h_2)} \phi_2  -\textrm{ a state on } \Pol(\QG_1 \widehat{*}\, \QG_2).
\end{equation}
Recall Theorem \ref{repfree}, where we described the representation theory of $\QG_1 \widehat{*}\, \QG_2$.

\begin{lem} \label{lem:diamondconvolution}
Let $\QG_1$, $\QG_2$ be compact quantum groups and let $\phi_1$, $\phi_2$ be states on, respectively, $\Pol(\QG_1)$ and $\Pol(\QG_2)$. Then
$\phi_1 \diamond \phi_2$ satisfies (and is determined by) the formulas
\begin{gather}
\label{diam1} (\phi_1 \diamond \phi_2) (1) = 1, \\
\label{diam2} (\phi_1 \diamond \phi_2) \big(\big(U^{\alpha_1}\tp \cdots \tp U^{\alpha_k}\big)_{(l_1,\ldots,l_k),(m_1,\ldots,m_k)} \big) = \phi_{i(1)} (u^{\alpha_1}_{l_1,m_1}) \cdots \phi_{i(k)} (u^{\alpha_1}_{l_k,m_k})
\end{gather}
for any $k\in \bn$,  $i(j)\in \{1,2\}, i(j)\neq i(j+1), \alpha_{j} \in \Irr_{\QG_{i(j)}}$, $\alpha_j \neq 1 $, $l_j,m_j \in \{1,\ldots, n_{\alpha_j}\}$.
Moreover if $\omega_1$, $\omega_2$ are two further states on, respectively, $\Pol(\QG_1)$ and $\Pol(\QG_2)$, then we have
\begin{equation} \label{diam3}(\phi_1 \diamond \phi_2 ) \star (\omega_1 \diamond \omega_2 ) = (\phi_1 \star \omega_1 ) \diamond (\phi_2 \star \omega_2 ), \end{equation}
where $\star$ above denotes the convolution product of functionals.
\end{lem}
\begin{proof}
The fact that the formulas \eqref{diam1} and \eqref{diam2} determine $\phi_1 \diamond \phi_2$ uniquely follows immediately from Theorem~\ref{repfree}. To show that these formulas hold it suffices to observe that for $i=1,2$, a nontrivial  $\alpha \in \Irr_{\QG_i}$ and any $l,m\in \{1,\ldots,n_\alpha\}$, we have that $h_i(u^{\alpha}_{l,m}) = 0$, and then use the definition of the conditionally free product.

The second part of the proof is then an explicit check of the equality in \eqref{diam3} on the elements of the form appearing in \eqref{diam2} (recall that they span
$\Pol(\QG_1 \widehat{*}\, \QG_2)$), based on applying the fact that the coproduct acts on the entries of a finite-dimensional unitary representation as `matrix multiplication':
$\Com(u_{i,j}) = \sum_{k=1}^{n_U} u_{i,k} \ot u_{k,j}$.
\end{proof}

Recall the definition of a convolution semigroup of states on a locally compact quantum group, Definition \ref{convssDef}. The last lemma yields the following result, which can be interpreted as providing a method of constructing quantum L\'evy processes (\cite{QSCC2}) on dual free products of compact quantum groups.

\begin{tw} \label{diamondsemigroup}
Let $\QG_1$, $\QG_2$ be compact quantum groups equipped, respectively, with convolution semigroups of states $(\phi_t)_{t\geq 0}$ and $(\omega_t)_{t\geq 0}$. Then $(\phi_t\diamond \omega_t)_{t\geq 0}$ is a convolution semigroup of states on $C^u(\QG_1 \widehat{*}\, \QG_2)$.

Moreover, if $L_1:\Pol(\QG_1) \to \bc$ and $L_2:\Pol(\QG_1) \to \bc$ are generating functionals of, respectively, $(\phi_t)_{t\geq 0}$ and $(\omega_t)_{t\geq 0}$, then the generating functional of $(\phi_t\diamond \omega_t)_{t\geq 0}$ is determined by the formula
\begin{align}
L\big(\big(U^{\alpha_1}\tp \cdots \tp   &   U^{\alpha_k}\big)_{(l_1,\ldots,l_k),(m_1,\ldots,m_k)} \big) = \nonumber\\ & \sum_{j=1}^k
\delta_{l_1,m_1}\cdots \delta_{l_{j-1}, m_{j-1}} L_{i(j)}(u^{\alpha_j}_{l_j,m_j}) \delta_{l_{j+1}, m_{j+1}} \cdots \delta_{l_k,m_k}, \label{gendiam}
\end{align}
again  for any $k\in \bn$,  $i(j)\in \{1,2\}, i(j)\neq i(j+1), \alpha_{j} \in \Irr_{\QG_{i(j)}}$, $\alpha_j \neq 1 $, $l_j,m_j \in \{1,\ldots, n_{\alpha_j}\}$.
\end{tw}

\begin{proof}
For each $t\geq 0$, define $\mu_t:=\phi_t\diamond \omega_t$ and consider the family of states $(\mu_t)_{t\geq 0}$. By Lemma \ref{lem:diamondconvolution}, this family satisfies the first property in Definition \ref{convssDef}. Further, as for any finite-dimensional unitary representation $U=(u_{i,j})_{i,j=1}^{n_U}$ of a compact quantum group we have that $\epsilon(u_{i,j}) = \delta_{i,j}$, the formulas \eqref{diam1} and \eqref{diam2} imply that we have $\epsilon_{\QG_1} \diamond\, \epsilon_{\QG_2} = \epsilon_{\QG_1 \widehat{*}\, \QG_2}$, so that $\mu_0 = \epsilon$. Finally, as each $\mu_t$ is a state, it suffices to check the convergence in part (iii) of Definition \ref{convssDef} on the elements of the type appearing in \eqref{diam2}, where it follows from the formulas describing the respective counits.

Taking once again into account the equality \eqref{diam2}, the formula \eqref{gendiam} is obtained from the Leibniz rule, as $L$ is determined by the formula (\ref{e6.24}).
\end{proof}

We return now to the main point of this section.

\begin{tw} \label{freeHAP}
Let $\QG_1$, $\QG_2$ be discrete quantum groups. Then their free product $\QG_1* \QG_2$ has the Haagerup property if and only if both $\QG_1$ and $\QG_2$ have the Haagerup property.
\end{tw}

\begin{proof}
If $\QG_1* \QG_2$ has the Haagerup property, then both $\QG_1$ and $\QG_2$ have the Haagerup property by Propositions \ref{HAP:subgroups} and \ref{subgroupfree} (note that discrete quantum groups are automatically coamenable).

Conversely, let $(\phi^1_k)_{k\in \bn}$ and $(\phi^2_k)_{k\in \bn}$ be sequences of states on, respectively, $\Pol(\widehat{\QG_1})$ and $\Pol(\widehat{\QG_2})$  satisfying the conditions in Lemma~\ref{lem:boundedperturbstates}.  For $k\in\bn$, let $\mu_k =
\phi^1_k \diamond \phi^2_k$, a state on $\Pol(\widehat{\QG_1} \widehat{*}\, \widehat{\QG_2})$.  Note that the formula \eqref{diam2} interpreted matricially says that for any $l\in \bn$,  $i(j)\in \{1,2\}, i(j)\neq i(j+1), \alpha_{j} \in \Irr_{\QG_{i(j)}}$, $\alpha_j \neq 1 $,
\begin{equation} (\mathcal F\mu_k)^{\alpha_1\cdots \alpha_l} = \bigotimes_{j=1}^l (\mathcal F\phi^{i(j)}_k)^{\alpha_j}.
\end{equation}
It is then elementary to check that the sequence  $(\mu_k)_{k\in \bn}$ satisfies the conditions of Lemma~\ref{lemma:hap_dis_case} -- the fact that the respective matrices belong to $\bigoplus_{\beta \in \Irr_{\widehat{\QG_1} \widehat{*}\, \widehat{\QG_2}}} M_{n_{\beta}}$ follows from the fact that $\|\mu_k^{\alpha_1\cdots \alpha_l}\|\leq \exp(-\frac{l}{k})$.
\end{proof}

\begin{rem}
We could  also prove the backward implication of the above theorem using Lemma \ref{lem:boundedperturb} and the equivalence of conditions (i) and (iii) in Theorem \ref{generators}.  Indeed, let $L_1:\Pol(\widehat{\QG_1})\to \bc$ and  $L_2:\Pol(\widehat{\QG_2})\to \bc$ be proper symmetric generating functionals as in Lemma~\ref{lem:boundedperturb}. Denote the convolution semigroups of states associated with $L_1$ and $L_2$ via Lemma~\ref{convgen} by, respectively, $(\phi_t)_{t\geq 0}$ and $(\omega_t)_{t\geq 0}$ and let $L:\Pol(\widehat{\QG_1} \widehat{*}\, \widehat{\QG_2})\rightarrow\mathbb C$ be the generator of the convolution semigroup of states $(\phi_t\diamond \omega_t)_{t\geq 0}$. Then using arguments similar to those in the proof of Theorem~\ref{freeHAP} and exploiting Theorem~\ref{diamondsemigroup} one can show that $L$ is a proper symmetric generating functional.
\end{rem}

\begin{rem}
Note that recently A.\,Freslon showed in \cite{Fres} that \emph{weak amenability} (with the Cowling-Haagerup constant equal 1) is preserved under taking free products of discrete quantum groups, extending a result of E.\,Ricard and Q.\,Xu for discrete groups (\cite{RicXu}).
\end{rem}

\begin{exam}
Theorem \ref{freeHAP} offers a method of constructing non-amenable, non-unimodular discrete quantum groups with the Haagerup property: it suffices to take the free product of a non-amenable discrete quantum group with the Haagerup property (such as for example $\widehat{U_N^+}$ for $N\geq2$, see \cite{Brannan}) and a non-unimodular amenable discrete quantum group (such as for example $\widehat{SU_q(2)}$ for $q\in (0,1)$). Another path to the construction of such examples leads via monoidal equivalence, as observed in Section 6 of the article \cite{Fres2}. Note that the notion of the Haagerup property for quantum groups which seems to be implicitly considered in \cite{Fres2}, formed in term of the existence of suitable state-induced multipliers on $C^u(\hh \G)$, is equivalent to the one studied here, as follows from Theorem \ref{equivHAP}. The main result of \cite{Fres2} has very recently been extended  in \cite{DCY}, by other techniques, to duals of all free orthogonal quantum groups.
\end{exam}

For the rest of the section we return to the context of discrete unimodular quantum groups.

\begin{rem}
As mentioned in the introduction to this section, the more difficult implication of Theorem \ref{freeHAP} for discrete \emph{unimodular} quantum groups can be proved via Theorem \ref{groupalgKac}. Indeed, assume that $\QG_1$, $\QG_2$  are discrete unimodular quantum groups with the Haagerup property.
Then both finite von Neumann algebras $L^{\infty}(\widehat{\QG_1})$ and $L^{\infty}(\widehat{\QG_2})$ have the Haagerup approximation property. By Theorem 2.3 of \cite{j} (see also \cite{Bo93}) so does the finite von Neumann algebra $L^{\infty}(\widehat{\QG_1}) *  L^{\infty}(\widehat{\QG_2}) \cong L^{\infty}(\widehat{\QG_1* \QG_2})$.
\end{rem}

Using an extension of the method described in the  above remark  and the results of Boca from \cite{Bo93} we can prove the following results.  The first relates to the free products of unimodular discrete quantum groups with amalgamation over a finite quantum subgroup (\cite{Ve04}) and the second to HNN extensions of unimodular discrete quantum groups (\cite{Fi12}).  Note that both these constructions leave the class of unimodular discrete quantum groups invariant, as can be deduced from the explicit formulas for the Haar states of the dual quantum groups, which imply that these Haar states remain tracial (we refer for the details to \cite{Ve04} and \cite{Fi12}). The corresponding results for classical groups can be found for example in \cite{book}.

\begin{prop} \label{amalgam}
If $\G_1$ and $\G_2$ are unimodular discrete quantum groups with a common finite closed quantum subgroup $\mathbb{H}$, then the amalgamated free product $\G=\G_1\underset{\mathbb{H}}{*}\G_1$ (see \cite{Ve04}) has the  Haagerup property if and only if both $\G_1$ and $\G_2$ have the  Haagerup property.
\end{prop}

\begin{proof} By the results of \cite{Ve04}, the dual von Neumann algebra $L^{\infty}(\widehat{\G})$ is isomorphic to $\mc M_1\underset{B}{*} \mc M_2$ where $\mc M_i=L^{\infty}(\widehat{\G_i})$ for $i=1,2$ and $\mc B=L^{\infty}(\widehat{\mathbb{H}})$. Since $B$ is finite dimensional, it follows from the results of \cite{Bo93} that $\mc M_1\underset{\mc B}{*} \mc M_2$ has the (von Neumann algebraic) Haagerup approximation property if and only if both $\mc M_1$ and $\mc M_2$ have the Haagerup approximation property.
\end{proof}

\begin{prop} \label{HNN}
If $\G$ is a discrete quantum group with a finite closed quantum subgroup $\mathbb{H}$, so that $\theta\,:\,C(\widehat{\mathbb{H}})\rightarrow C(\widehat{\G})$ is an injective unital $*$-homomorphism which intertwines the comultiplications, then the HNN extension HNN$(\G,\mathbb{H},\theta)$ (see \cite{Fi12}) has the  Haagerup property if and only if $\G$ has the  Haagerup property.
\end{prop}

\begin{proof} From the computation of the Haar state in \cite{Fi12} we know that the dual von Neumann algebra of the HNN extension is equal to ${\rm HNN}(\mc M,\mc N,\theta)$ where $\mc M=L^{\infty}(\widehat{\mathbb{G}})$ and $\mc N=L^{\infty}(\widehat{\mathbb{H}})$ and $\theta$ is the induced unital normal *-homomorphism at the von Neumann algebraic level (see \cite{FV12} for the HNN construction of von Neumann algebras). By \cite[Remark 4.6]{FV12} ${\rm HNN}(\mc M,\mc N,\theta)$ is isomorphic to a von Neumann algebra of the form $p(M_2(\bc)\ot \mc M)\underset{N\oplus N}{*}(M_2(\bc)\ot \mc N)p$ which has the Haagerup approximation property whenever $\mc M$ has the Haagerup approximation property (by the results of \cite{Bo93} and \cite{j}, Theorem 2.3 (i)). The other implication follows from Proposition \ref{HAP:subgroups}, since $\G$ is a closed quantum subgroup of the HNN extension in question.
\end{proof}

\end{document}